\documentclass[bj,preprint]{imsart}

\RequirePackage[OT1]{fontenc}
\RequirePackage{amsthm,amsmath}
\RequirePackage[colorlinks,citecolor=blue,urlcolor=blue]{hyperref}
\RequirePackage[numbers]{natbib}
\usepackage{amsmath}
\usepackage{amssymb}
\usepackage{amsthm}
\usepackage{color}
\usepackage{amsfonts}
\usepackage{bbm}

\arxiv{arXiv:1405.0671}

\startlocaldefs

\newcommand{\mf}{\mathcal{F}}
\newcommand{\mg}{\mathcal{G}}
\newcommand{\mmp}{\mathbb{P}}
\newcommand{\Cov}{\mathrm{Cov}}
\DeclareMathOperator{\Var}{\mathrm{Var}}
\DeclareMathOperator{\sign}{\mathrm{sign}}

\newcommand{\dod}{\overset{\mathrm{d}}{\to}}
\newcommand{\fdc}{\overset{\mathrm{f.d.}}{\Rightarrow}}

\newcommand{\tp}{\overset{P}{\to}}

\newcommand{\me}{\mathbb{E}}

\newcommand{\mr}{\mathbb{R}}
\newcommand{\mn}{\mathbb{N}}
\newcommand{\F}{\mathcal{F}}

\newcommand{\lit}{\underset{t\to\infty}{\lim}}

\newcommand{\dy}{\mathrm{d} \mathit{y}}

\newcommand{\ds}{\mathrm{d} \mathit{s}}

\newcommand{\comp}{\mathsf{c}}
\newcommand{\1}{\mathbbm{1}}

\newtheorem{thm}{Theorem}[section]
\newtheorem{lemma}[thm]{Lemma}

\newtheorem{cor}[thm]{Corollary}
\newtheorem{assertion}[thm]{Proposition}
\theoremstyle{definition}
\newtheorem{define}[thm]{Definition}
\newtheorem{example}[thm]{Example}
\theoremstyle{remark}
\newtheorem{rem}[thm]{Remark}

\endlocaldefs

\begin{document}

\begin{frontmatter}

\title{Asymptotics of random processes with immigration I: scaling limits}

\runtitle{Asymptotics of random processes with immigration I}

\begin{aug}

\author{\fnms{Alexander} \snm{Iksanov}\thanksref{a,e1}\corref{}\ead[label=e1,mark]{iksan@univ.kiev.ua}}
\author{\fnms{Alexander} \snm{Marynych}\thanksref{a,e2}\ead[label=e2,mark]{marynych@unicyb.kiev.ua}}
\and
\author{\fnms{Matthias} \snm{Meiners}\thanksref{b,e3}\ead[label=e3,mark]{meiners@mathematik.tu-darmstadt.de}}

\runauthor{A. Iksanov et al.}

\affiliation{Taras Shevchenko National University of Kyiv and Technical University of Darmstadt}
\address[a]{Faculty of Cybernetics, Taras Shevchenko National University of Kyiv, 01601 Kyiv, Ukraine. \printead{e1,e2}}
\address[b]{Fachbereich Mathematik, Technische Universitat Darmstadt, 64289 Darmstadt, Germany. \printead{e3}}

\end{aug}

\begin{abstract}
Let $(X_1, \xi_1), (X_2,\xi_2),\ldots$ be i.i.d.~copies
of a pair $(X,\xi)$ where $X$ is a random process with paths in
the Skorokhod space $D[0,\infty)$ and $\xi$ is a positive random
variable. Define $S_k := \xi_1+\ldots+\xi_k$, $k \in \mn_0$ and
$Y(t) := \sum_{k\geq 0} X_{k+1}(t-S_k) \1_{\{S_k \leq t\}}$,
$t\geq 0$. We call the process $(Y(t))_{t \geq 0}$ random process
with immigration at the epochs of a renewal process.
We investigate weak convergence of the finite-dimensional distributions
of $(Y(ut))_{u>0}$ as $t\to\infty$.
Under the assumptions that the covariance function of $X$ is regularly varying in
$(0,\infty)\times (0,\infty)$ in a uniform way,
the class of limiting processes is rather rich and includes Gaussian processes
with explicitly given covariance functions, fractionally
integrated stable L\'evy motions and their sums when the law of
$\xi$ belongs to the domain of attraction of a stable law with
finite mean, and conditionally Gaussian processes with explicitly
given (conditional) covariance functions, fractionally integrated
inverse stable subordinators and their sums when the law of $\xi$
belongs to the domain of attraction of a stable law with infinite
mean.
\end{abstract}




\end{frontmatter}

\section{Introduction}
\subsection{Random processes with immigration at the epochs of a renewal process}
Denote by $D[0,\infty)$ and $D(0,\infty)$ the Skorokhod spaces of
right-continuous real-valued functions which are defined on
$[0,\infty)$ and $(0,\infty)$, respectively, and have finite
limits from the left at each point of the domain. Throughout the
paper, we abbreviate $D[0,\infty)$ by $D$. Let
$X:=(X(t))_{t\in\mr}$ be a random process with paths in $D$
satisfying $X(t)=0$ for all $t<0$, and let $\xi$ be a positive
random variable. Arbitrary dependence between $X$ and $\xi$ is
allowed. It is worth stating explicitly that we do not exclude the
possibility $X=h$ a.s.~for a deterministic function $h$.

Further, let $(X_1,\xi_1), (X_2, \xi_2),\ldots$ be i.i.d.~copies of the pair $(X,\xi)$
and denote by $(S_k)_{k \in \mn_0}$ the zero-delayed random walk with increments $\xi_j$, that is,
\begin{equation*}
S_0 := 0, \qquad    S_k:=\xi_1+\ldots+\xi_k, \quad  k\in\mn.
\end{equation*}
We write $(\nu(t))_{t\in\mr}$ for the corresponding first-passage time process, i.e.,
\begin{equation*}
\nu(t):=\inf\{k\in\mn_0: S_k>t\}=\#\{k\in\mn_0: S_k\leq t\},
\quad  t\in\mr,
\end{equation*}
where the last equality holds a.s. The process $Y :=
(Y(t))_{t\in\mr}$ defined by
\begin{equation}    \label{eq:Y(t)}
Y(t) := \sum_{k\geq 0}X_{k+1}(t-S_k) = \sum_{k=0}^{\nu(t)-1}X_{k+1}(t-S_k), \quad   t\in\mr
\end{equation}
will be called {\it random process with immigration at the epochs
of a renewal process} or, for short, {\it random process with
immigration}. The interpretation is as follows: at time $S_0=0$
the immigrant $1$ starts running a random process $X_1$, for $k\in\mn$, at
time $S_k$ the immigrant $k+1$ starts running a random process $X_{k+1}$,
$Y(t)$ being the sum of all processes run by the immigrants up to
and including time $t$. We advocate using this term for two
reasons. First, we believe that it is more informative than the
more familiar term {\it renewal shot noise process with random
response functions $X_k$}; in particular, the random process $Y$
defined by \eqref{eq:Y(t)} has little in common with the
originally defined shot noise processes \citep{Schottky:18}
intended to model the real shot noise in vacuum tubes which were
based on Poisson inputs and deterministic response functions.
Second, the new term was inspired by the fact that if $X$ is a
continuous-time branching process, then $Y$ is known in the
literature as a {\it branching process with immigration}.

Random processes with immigration have been used to model
various phenomena. An incomplete list of possible areas of
applications includes anomalous diffusion in physics
\citep{Metzler+Klafter:2000}, earthquakes occurrences in geology
\citep{Vere-Jones:1970}, rainfall modeling in meteorology
\citep{Rodriguez-Iturbe+Cox+Isham:1987,Waymire+Gupta:1981}, network
traffic in computer sciences \citep{Konstantopoulos+Lin:98,
Mikosch+Resnick:2006, Resnick+Rootzen:2000, Resnick+Berg:2000} as
well as insurance
\citep{Klueppelberg+Mikosch:1995a,Klueppelberg+Mikosch:1995b} and
finance \citep{Klueppelberg+Kuhn:2004, Samorodnitsky:1996}. Further
references concerning mainly renewal shot noise processes can be found in
\citep{Alsmeyer+Iksanov+Meiners:2015,Hsing+Teugels:1989,Iksanov+Marynych+Meiners:2014,Vervaat:1979}.
Although we do not have any particular application in mind, our
results are potentially useful for either of the aforementioned fields.

\subsection{Weak convergence of random processes with immigration}
The paper at hand is part of a series of papers that further contains the references
\citep{Iksanov:2013,Iksanov+Marynych+Meiners:2014,Iksanov+Marynych+Meiners:2015II}
in which we investigate the
asymptotic distribution of $Y$.
When $\mu := \me \xi < \infty$ and $X(t)$ tends to $0$ quickly as $t \to \infty$
(more precisely, if $t \mapsto \me[|X(t)| \wedge 1]$ is a directly Riemann integrable function),
then, under mild technical assumptions, $(Y(u+t))_{u \geq 0}$ converges to a stationary version of the process.
This convergence is investigated in \citep{Iksanov+Marynych+Meiners:2015II}.
In the present paper, we focus on the case where the law of $\xi$ is in the domain of attraction of a stable law of index $\alpha \not = 1$
and, if $\mu < \infty$ (equivalently, $\alpha > 1$), either $\me[X(t)]$ or $\Var[X(t)]$ is too large for convergence to stationarity.
In this situation, we investigate the weak convergence of the finite-dimensional distributions of
$Y_t(u):= a(t)^{-1}(Y(ut)-b(ut))$ with suitable norming constants $a(t)>0$ and shifts $b(t)\in\mr$.
This convergence is mainly regulated by two factors:
the tail behavior of $\xi$ and the asymptotics of
the finite-dimensional distributions of $X(t)$ as $t\to\infty$. The
various combinations of these give rise to a broad spectrum of
possible limit results.
In this paper, assuming that $h(t):=\me [X(t)]$ is finite for all $t\geq 0$, we start with the decomposition
\begin{equation}\label{decomp}
Y(t)-b(t) =   \bigg(\!Y(t)-\sum_{k\geq 0} h(t-S_k)\1_{\{S_k\leq
t\}} \!\bigg) + \bigg(\!\sum_{k\geq 0} h(t-S_k)\1_{\{S_k \leq
t\}}-b(t) \!\bigg)
\end{equation}
and observe that $Y_t(u)$ may converge if at least one summand in
\eqref{decomp}, properly normalized, converges weakly.

The asymptotic behavior of the second summand, properly normalized, is
driven by the functional limit theorems for the first-passage time
process $(\nu(t))_{t \geq 0}$ as well as the behavior of the function $h$ at
infinity.

The asymptotics of the first summand, properly
normalized, is accessible via martingale central limit theory
or convergence results for triangular arrays.
When $\me\xi$ is finite, the normalizing constants and limiting processes for the
first summand are completely determined by properties of $X$, the
influence of the law of $\xi$ is small. This phenomenon can easily
be understood: the randomness induced by the $\xi_k$'s is governed
by the law of large numbers for $\nu(t)$ and is thus
degenerate in the limit. When $\me \xi$ is infinite and
$\mmp\{\xi>t\}$ is regularly varying with index larger than $-1$,
$\nu(t)$, properly normalized, weakly converges to a
non-degenerate law. Hence, unlike the finite-mean case, the
randomness induced by $\xi$ persists in the limit.

It turns out that there are situations in which one of the
summands in \eqref{decomp} dominates (cases $p=0$ and $p=1$ of
Theorem \ref{Thm:mu<infty}; cases $q=0$ and $q=1$ of Theorem
\ref{Thm:mu=infty}; the case where $h \equiv 0$), and those in
which the contributions of the summands are comparable (case
$p\in(0,1)$ of Theorem \ref{Thm:mu<infty} and case $q\in(0,1)$ of
Theorem \ref{Thm:mu=infty}). A nice feature of the former
situation is that possible dependence of $X$ and $\xi$ gets
neutralized by normalization (provided $\lim_{t\to\infty}a(t)=+\infty$) so that the limit
results are only governed by individual contributions of $X$ and
$\xi$. Suppose, for the time being, that the latter situation
prevails, i.e., the two summands in \eqref{decomp}
are of the same order, and that $X$ and $\xi$ are independent. From the
discussion above it should be clear that whenever $\me\xi$ is
finite, the two limit random processes corresponding to the
summands in \eqref{decomp} are independent,
whereas this is not the case, otherwise. Still, we are able to show that the summands in
\eqref{decomp} converge jointly.

When $X$ and $\xi$ are dependent,
proving such a joint convergence remains an open problem.
In the particular case where $X(t)=\1_{\{|\log(1-W)|>t\}}$ and $\xi=|\log
W|$ for some random variable $W \in (0,1)$ a.s., this problem,
already reported in Section 1 of
\citep{Iksanov+Marynych+Vatutin:2013+}, turned out to be the major
obstacle on the way towards obtaining the description of {\it all}
possible modes of weak convergence of the number of empty boxes in
the Bernoulli sieve.

Adequacy of the aforementioned approach was realized by the
authors some time ago, and as a preparation for its implementation
the articles \citep{Iksanov:2013, Iksanov+Marynych+Meiners:2014}
were written. In the first of these papers, functional limit
theorems for the second summand have been established in the case
where $h$ is eventually increasing\footnote{We call a function $h$
increasing (decreasing) if $s<t$ implies $h(s) \leq h(t)$
(resp.,~$h(s) \geq h(t)$) and strictly increasing (decreasing) if $s<t$ implies
$h(s)<h(t)$ (resp.,~$h(s)>h(t)$).}, while in the second
convergence of the finite-dimensional distributions of the second
summand has been proved in the case where $h$ is eventually
decreasing\footnote{The present paper does not offer new results
about weak convergence of the second summand in \eqref{decomp}
alone. However, the joint convergence of the summands in
\eqref{decomp} is treated here for the first time.}.

\subsection{Bibliographic comments and known results}
In the case where $\xi$ has an exponential law,
the process $Y$ (or its stationary version) is a {\it Poisson shot noise process}.
Weak convergence of Poisson shot noise processes has received considerable attention.
In some papers of more applied nature weak convergence of $Y_t(u)$ for $X$ having a
specific form is investigated. In the list to be given next $\eta$ denotes a
random variable independent of $\xi$ and $f$ a deterministic function
which satisfies certain restrictions which are specified in the cited papers:
\begin{itemize}
\item
$X(t) = \1_{\{\eta>t\}}$ and $X(t) = t \wedge \eta$, functional convergence, see \citep{Resnick+Rootzen:2000};
\item
$X(t)=\eta f(t)$, stationary version of $Y$, functional convergence, see \citep{Klueppelberg+Kuhn:2004};
\item
$X(t)=f(t\wedge\eta)$, convergence of finite-dimensional distributions, see \citep{Konstantopoulos+Lin:98};
functional convergence, see \citep{Resnick+Berg:2000};
\item
$X(t)=\eta_1/\eta_2 f(t\eta_2)$, stationary version,
convergence of finite-dimensional distributions, see \citep{Giraitis+Surgailis:1990, Giraitis+Surgailis:1991}.
\end{itemize}
The articles \citep{Heinrich+Schmidt:1985, Jedidi_etal:2012, Klueppelberg+Mikosch:1995a, Klueppelberg+Mikosch+Schaerf:2003,
Lane:1984} are of more theoretical nature, and study weak
convergence of $Y_t(u)$ for general (not explicitly specified) $X$.
The work \citep{Jedidi_etal:2012} contains further pointers to relevant literature which could have extended our list of
particular cases given above.

In the case where the law of $\xi$ is exponential, the variables
$Y_t(u)$ have infinitely divisible laws with characteristic functions of a rather simple form.
Furthermore, the convergence, as $t \to \infty$, of these characteristic functions to a
characteristic function of a limiting infinitely divisible law follows from the general theory.
Also, in this context Poisson random measures arise naturally and working with them considerably simplifies the analysis.
In the cases where the law of $\xi$ is not exponential,
the aforementioned approaches are not applicable. We are aware of several papers  in which weak convergence of processes $Y$,
properly normalized, centered and rescaled, is investigated
in the case where $\xi$ has distribution other than exponential.
Iglehart \citep{Iglehart:1973} has proved weak convergence of
\begin{equation*}
\frac{1}{\sqrt{n}} \bigg(\sum_{k\geq 0}X_{k+1}(u-n^{-1}S_k)\1_{\{S_k\leq nu\}}- \frac{n}{\me\xi} \int_0^u\me [X(y)] \dy\bigg)
\end{equation*}
in $D[0,1]$ to a Gaussian process, as $n\to\infty$, under rather restrictive assumptions
(in particular, concerning the existence of moments of order four).
See also Theorem 1 on p.~103 of \citep{Borovkov:1984} for a similar result with
$X(t)=\1_{\{\eta>t\}}$ in a more general setting.
For $X(t)=\int_0^t f(s,\eta) \ds$, weak convergence of $(Y_t(u))_{0 \leq u \leq 1}$ on $D[0,1]$ as $t \to \infty$
was established in \citep{Iglehart+Kennedy:1970} under the assumptions that $\xi$ and $\eta$ are independent,
that $\int_0^\infty |f(s,x)| \ds <\infty$ for every $x\in\mr$ and some other conditions.
For $X(t)=\1_{\{\eta>t\}}$, weak convergence of finite-dimensional distributions of $(Y_t(u))_u$, as $t\to\infty$,
has been settled in \citep{Mikosch+Resnick:2006} under the assumption that
$\xi$ and $\eta$ are independent and some moment-type conditions.

Last but not least, weak convergence of
$Y_t(1)$ has been much investigated, especially in the
case where $X$ is a branching process (see, for instance,
\citep{Alsmeyer+Slavtchova-Bojkova:2005,Jagers:1968,Pakes+Kaplan:1974}).

\subsection{Additional definitions}
Throughout the paper we assume that $h(t):=\me [X(t)]$ is finite for all $t\geq 0$ and that the covariance
\begin{equation*}
f(s,t) := \Cov[X(s),X(t)] = \me[X(s)X(t)] - \me[X(s)] \me[X(t)]
\end{equation*}
is finite for all $s,t \geq 0$. The variance of $X$ will be
denoted by $v$, i.e., $v(t):=f(t,t)=\Var[X(t)]$. In what follows
we assume that $h, v \in D$. By Lebesgue's dominated convergence
theorem, local uniform integrability of $X^2$ is sufficient for
this to be true since the paths of $X$ belong to $D$. $h, v \in D$
implies that $h$ and $v$ are a.e.~continuous and locally bounded.
Consequently, $\int_0^t h(y)\dy$ and $\int_0^t v(y)\dy$ are well-defined as Riemann integrals.

\paragraph{Regular variation in $\mr^2$.}
Recall that a positive measurable function $\ell$, defined on some
neighborhood of $\infty$, is called {\it slowly varying} at
$\infty$ if $\lim_{t\to\infty} \frac{\ell(ut)}{\ell(t)}=1$ for all $u> 0$, see \cite[p.~6]{Bingham+Goldie+Teugels:1989}.
\begin{define}\label{regular_variation_R^2}
A function $r: [0,\infty)\times [0,\infty)\to\mr$ is {\it
regularly varying}\footnote{The canonical definition of the
regular variation in $\mr^2_+$ (see, for instance,
\citep{Haan+Resnick:1979}) requires nonnegativity of $r$.} in
$\mr^2_+:=(0,\infty)\times (0,\infty)$ if there exists a function
$C: \mr^2_+ \to (0,\infty)$, called {\it limit function}, such that
\begin{equation*}
\lim_{t \to \infty} {r(ut,wt)\over r(t,t)}=C(u,w), \quad u,w>0.
\end{equation*}
\end{define}
The definition
implies that $r(t,t)$ is  regularly varying at $\infty$, i.e.,
$r(t,t) \sim t^\beta \ell(t)$ as $t\to\infty$ for some $\ell$
slowly varying at $\infty$ and some $\beta\in\mr$
which is called the {\it index of regular variation}.
In particular, $C(a,a)=a^{\beta}$ for all $a>0$ and further
$$C(au,aw)=C(a,a)C(u,w)=a^\beta C(u,w)$$ for all $a,u,w>0$.
\begin{define}\label{regular_variation_R^21}
A function $r: [0,\infty)\times [0,\infty)\to\mr$ will be called
{\it fictitious regularly varying} of index $\beta$ in $\mr^2_+$
if
\begin{equation*}
\lim_{t \to \infty} {r(ut,wt)\over r(t,t)}=C(u,w),\quad u,w>0,
\end{equation*}
where $C(u,u):=u^\beta$ for $u>0$ and $C(u,w):=0$ for $u,w>0$,
$u\neq w$. A function $r$ will be called {\it wide-sense regularly
varying} of index $\beta$ in $\mr^2_+$ if it is either regularly
varying or fictitious regularly varying of index $\beta$ in
$\mr^2_+$.
\end{define}
The function $C$ corresponding to a fictitious regularly varying
function will also be called {\it limit function}.
\begin{define}\label{uniform_regular_variation_in_strips}
A function $r: [0,\infty)\times [0,\infty)\to\mr$ is {\it
uniformly regularly varying of index $\beta$ in strips in}
$\mr^2_+$ if it is regularly varying of index $\beta$ in $\mr^2_+$
and
\begin{equation}\label{0040}
\lim_{t \to \infty} \sup_{a\leq u\leq b} \bigg|\frac{r\big(ut,(u+w)t\big)}{r(t,t)}-C(u,u+w)\bigg| ~=~ 0
\end{equation}
for every $w>0$ and all $0<a<b<\infty$.
\end{define}

\paragraph{Limit processes for $Y_t(u)$.} The processes introduced in
Definition \ref{definition_v_process} arise as weak limits of the
first summand in \eqref{decomp} in the case $\me\xi<\infty$. We
shall check that they are well-defined at the beginning of
Section \ref{subsec:Proofs of Propositions}.
\begin{define}\label{definition_v_process}
Let $C$ be the limit function for a wide-sense regularly varying
function (see Definition \ref{regular_variation_R^21}) in
$\mr^2_+$ of
index $\beta$ for some $\beta\in (-1,\infty)$. 
We shall denote by $V_{\beta}:=(V_{\beta}(u))_{u> 0}$ a centered
Gaussian process with covariance function
\begin{equation*}
\me [V_\beta(u)V_\beta(w)] = \int_0^u C(u-y, w-y) \, \dy,
\quad   0< u \leq w,
\end{equation*}
when $C(s,t) \neq 0$ for some $s,t>0$, $s \neq t$, and a centered
Gaussian process with independent values and variance $\me
[V_\beta^2(u)] = (1+\beta)^{-1}u^{1+\beta}$, otherwise.
\end{define}

Definition \ref{inverse_subordinator_definition} reminds of the
notion of an inverse subordinator.
\begin{define}\label{inverse_subordinator_definition}
For $\alpha\in (0,1)$, let $W_\alpha:=(W_\alpha(t))_{t\geq 0}$ be
an $\alpha$-stable subordinator (nondecreasing L\'{e}vy process)
with Laplace exponent $-\log \me [\exp(-zW_\alpha(t))] =
\Gamma(1-\alpha) tz^\alpha$, $z\geq 0$, where $\Gamma(\cdot)$ is
the gamma function. {\it The inverse $\alpha$-stable subordinator}
$W_\alpha^\leftarrow:=(W_\alpha^\leftarrow(s))_{s\geq 0}$ is
defined by
$$W_\alpha^\leftarrow(s):=\inf\{t\geq 0: W_\alpha(t)>s\},\quad s\geq 0.$$
\end{define}
The processes introduced in Definition \ref{definition_process_z}
arise as weak limits of the first summand in \eqref{decomp} in the
case $\me\xi=\infty$. We shall check that these are well-defined
in Lemma \ref{integrals}.
\begin{define}\label{definition_process_z}
Let $W^\leftarrow_\alpha$ be an inverse $\alpha$-stable
subordinator and $C$ the limit function for a wide-sense regularly
varying function (see Definition \ref{regular_variation_R^21}) in
$\mr^2_+$ of index $\beta$ for some $\beta\in [-\alpha,\infty)$.
We shall denote by
$Z_{\alpha,\beta}:=(Z_{\alpha,\beta}(u))_{u> 0}$ a process
which, given $W_\alpha^\leftarrow$, is centered Gaussian with
(conditional) covariance
$$\me \big[Z_{\alpha,\beta}(u)Z_{\alpha,\beta}(w)\big|W^\leftarrow_\alpha\big]
= \int_{[0,u]}C(u-y,w-y) \, {\rm d}W_\alpha^\leftarrow(y),
\quad 0<u\leq w,$$
when $C(s,t)\neq 0$ for some $s,t>0$, $s \neq t$, and
a process which, given $W_\alpha$, is centered Gaussian with
independent values and (conditional) variance
$\me[Z_{\alpha,\beta}(u)^2|W^\leftarrow_\alpha]
= \int_{[0,u]}(u-y)^\beta \, {\rm d}W_\alpha^\leftarrow(y)$,
otherwise.
\end{define}

Throughout the paper, we use $\overset{\mathrm{d}}{\to}$,
$\overset{\mathbb{P}}{\to}$ and $\Rightarrow$ to denote weak
convergence of one-dimensional distributions, convergence in
probability and convergence in distribution in a function space, respectively.
Additionally, we write $Z_t(u) \fdc Z(u)$, $t \to \infty$ to
denote weak convergence of finite-dimensional distributions,
i.e., for any $n\in\mn$ and any $0<u_1<u_2<\ldots<u_n<\infty$,
\begin{equation*}
(Z_t(u_1),\ldots, Z_t(u_n)) ~\stackrel{\mathrm{d}}{\to}~ (Z(u_1),\ldots,Z(u_n)),    \quad   t\to\infty.
\end{equation*}
We stipulate hereafter that $\ell$, $\widehat{\ell}$ and
$\ell^\ast$ denote functions slowly varying at $\infty$ and that
all unspecified limit relations hold as $t\to\infty$.

\section{Main results} \label{sec:Main results}

\subsection{Asymptotic distribution of the first summand in \eqref{decomp}}

Proposition \ref{Prop:mu<infty} (case $\me\xi<\infty$) and
Proposition \ref{Prop:mu=infty} (case $\me\xi=\infty$) deal with
the asymptotics of the first summand in \eqref{decomp}.
\begin{assertion}    \label{Prop:mu<infty}
Assume that
\begin{itemize}
    \item
        $\mu := \me\xi\in (0,\infty)$;
    \item
        $f(u,w) = \Cov[X(u), X(w)]$ is either uniformly regularly varying in strips in
        $\mr_+^2$ or fictitious regularly varying in $\mr^2_+$, in
        either of the cases,
        of index $\beta$ for some $\beta \in (-1,\infty)$ and with limit function
        $C$; when $\beta=0$, there exists a positive monotone function $u$ satisfying
        $v(t)=\Var[X(t)]\sim u(t)$ as $t\to\infty$;
    \item
        for all $y>0$
        \begin{equation}    \label{eq:Lindeberg X-h}
        v_y(t) := \me\Big[(X(t)-h(t))^2 \1_{\{|X(t)-h(t)|>y\sqrt{tv(t)}\}}\Big] = o(v(t)), \quad	t\to\infty.
        \end{equation}
\end{itemize}
Then
\begin{equation}    \label{eq:1st summand convergence}
\frac{Y(ut)-\sum_{k\geq 0} h(ut-S_k)\1_{\{S_k\leq
ut\}}}{\sqrt{\mu^{-1}tv(t)}} ~\fdc~   V_\beta(u), \quad	t\to\infty
\end{equation}
where $V_\beta$ is a centered Gaussian process as introduced in
Definition \ref{definition_v_process}.
\end{assertion}

\begin{assertion}\label{Prop:mu=infty}
Assume that
\begin{itemize}
    \item
        $X$ is independent of $\xi$;
    \item
        for some $\alpha\in (0,1)$ and some $\ell^\ast$
\begin{equation}\label{eq:regular_variation_inf_mean}
        \mmp\{\xi>t\}\sim t^{-\alpha}\ell^\ast(t),\quad t\to\infty;
\end{equation}
    \item
        $f(u,w)=\Cov[X(u), X(w)]$ is either uniformly regularly varying in strips in
        $\mr_+^2$ or fictitious regularly varying in $\mr^2_+$, in
        either of cases,
        of index $\beta$ for some $\beta \in [-\alpha,\infty)$ and with limit function
        $C$; when $\beta=-\alpha$, there exists a positive increasing function $u$ with
        $\lim_{t \to \infty} \frac{v(t)}{\mmp\{\xi>t\}u(t)}=1$;
    \item
        for all $y>0$
        \begin{equation}    \label{eq:stable Lindeberg X-h}
        v_y(t) := \me\Big[(X(t)-h(t))^2 \1_{\{|X(t)-h(t)|>y\sqrt{v(t)/\mmp\{\xi>t\}}\}}\Big] = o(v(t)),\quad t \to
        \infty.
        \end{equation}
\end{itemize}
Then
\begin{equation*}
\sqrt{\frac{\mmp\{\xi>t\}}{v(t)}}\bigg(Y(ut)-\sum_{k\geq
0}h(ut-S_k)\1_{\{S_k\leq ut\}}\bigg) ~\fdc~ Z_{\alpha,\beta}(u),
\quad   t \to \infty,
\end{equation*}
where $Z_{\alpha,\beta}$ is a conditionally Gaussian process
as introduced in Definition \ref{definition_process_z}.
\end{assertion}

\begin{rem}
There is an interesting special case of Proposition
\ref{Prop:mu=infty} in which the finite-dimensional distributions
of $Y$ converge weakly, i.e., without normalization and centering.
Namely, if $h(t)\equiv 0$,
$\lim_{t \to \infty} v(t)/\mmp\{\xi>t\}=c$ for some $c>0$, and the assumptions of
Proposition \ref{Prop:mu=infty} hold (note that $\beta=-\alpha$
and one may take $u(t)\equiv c$), then \begin{equation*} Y(ut)
~\fdc~ \sqrt{c}Z_{\alpha,\alpha}(u).
\end{equation*}
\end{rem}

When $h(t)=\me [X(t)]$ is not identically zero, the centerings
used in Propositions \ref{Prop:mu<infty} and \ref{Prop:mu=infty} are random
which is undesirable.
Theorem \ref{Thm:mu<infty} (case $\me\xi<\infty$) and Theorem
\ref{Thm:mu=infty} (case $\me\xi=\infty$) stated below give limit
results with non-random centerings. These are obtained by
combining the results concerning weak convergence of the second
summand in \eqref{decomp} with Proposition \ref{Prop:mu<infty} and Proposition \ref{Prop:mu=infty}, respectively.

\subsection{Domains of attraction}

To fix notation for our main results, we recall here that the law of $\xi$ belongs to the domain of attraction
of a $2$-stable (normal) law if, and only if, either $\sigma^2 := \Var \xi < \infty$, or $\Var \xi=\infty$ and
\begin{equation}    \label{eq:slowly varying 2nd moment}
\me[\xi^2 \1_{\{\xi \leq t\}}] \sim \ell^\ast(t)
\end{equation}
for some $\ell^\ast$. Further, the law of $\xi$ belongs to the
domain of attraction of an $\alpha$-stable law, $\alpha \in (0,2)$
if, and only if,
\begin{equation}\label{eq:domain alpha}
\mmp\{\xi>t\} \sim t^{-\alpha}\ell^\ast(t)
\end{equation}
for some $\ell^\ast$. In the present paper, we do not treat the
case $\alpha=1$, for it is technically more complicated than the
others and does not shed any new light on weak convergence of
random processes with immigration.

If $\mu=\me\xi=\infty$, then necessarily $\alpha\in (0,1)$
(because we excluded the case $\alpha=1$) and according to
Corollary 3.4 in \citep{Meerschaert+Scheffler:2004} we have
\begin{equation}\label{weak}
\mathbb{P}\{\xi>t\}\nu(ut) ~\Rightarrow~ W^\leftarrow_\alpha(u)
\end{equation}
in the $J_1$-topology on $D$.

If $\mu<\infty$, then necessarily $\alpha\in (1,2]$ (where
$\alpha=2$ corresponds to the case of attraction to a normal law)
and according to Theorem 5.3.1 and Theorem 5.3.2 in
\citep{Gut:2009} or Section 7.3.1 in \citep{Whitt:2002} we have
\begin{equation}    \label{eq:FLT for N(t)}
\frac{\nu(ut)-\mu^{-1}ut}{\mu^{-1-1/\alpha}c(t)} ~\Rightarrow~
\mathcal{S}_\alpha(u),
\end{equation}
where
\begin{itemize}
\item
if $\sigma^2<\infty$, then $\mathcal{S}_2:=(\mathcal{S}_2(u))_{u
\geq 0}$ is a Brownian motion; $c(t)=\sigma \sqrt{t}$ and the
convergence takes place in the $J_1$-topology on $D$;
\item
if $\sigma^2=\infty$ and \eqref{eq:slowly varying 2nd moment}
holds, then $c(t)$ is some positive continuous function such that
\begin{equation*}
\lim_{t \to \infty} t \ell^\ast(c(t)) / c(t)^2 = 1,
\end{equation*}
and the convergence takes place in the $J_1$- topology on $D$;
\item if in \eqref{eq:domain alpha} $\alpha\in (1,2)$,
then $\mathcal{S}_\alpha:=(\mathcal{S}_\alpha(u))_{u \geq 0}$ is a
spectrally negative $\alpha$-stable L\'{e}vy process such that
$\mathcal{S}_\alpha(1)$ has the characteristic function
\begin{equation*}
\mathbb{E} [\exp(iz\mathcal{S}_\alpha(1))] = \exp\{-|z|^\alpha
\Gamma(1-\alpha)(\cos(\pi\alpha/2)+{\rm
i}\sign(z)\sin(\pi\alpha/2))\},\quad z\in\mr,
\end{equation*}
where $\Gamma(\cdot)$ denotes Euler's gamma function; $c(t)$ is
some positive continuous function satisfying
\begin{equation*}
\lim_{t \to \infty} t\ell^\ast(c(t))/c(t)^{\alpha}=1,
\end{equation*}
and the convergence takes place in the $M_1$-topology on $D$.
\end{itemize}
In any case, $c(t)$ is regularly varying at $\infty$ of index $1/\alpha$, see Lemma \ref{ct}.
We refer to \citep{Whitt:2002} for extensive information concerning
both the $J_1$- and $M_1$- convergence on $D$.

\subsection{Scaling limits of random processes with immigration}

\begin{thm} \label{Thm:mu<infty}
Assume that the law of $\xi$ belongs to the domain of attraction
of an $\alpha$-stable law, $\alpha\in (1,2]$, that $c$ is as in \eqref{eq:FLT for N(t)},
that $h$ is eventually monotone and not identically zero, and that the
following limit
$$
p:=\lim_{t\to\infty}\frac{c(t)^2h(t)^2}{\int_0^t v(y)\dy +
c(t)^2h(t)^2}\in [0,1],
$$
exists. Assume further that
\begin{itemize}
\item
if $p<1$, then the assumptions of Proposition \ref{Prop:mu<infty} hold;
\item
if $p>0$, then $h(t) \sim t^\rho\widehat{\ell}(t)$ as $t\to\infty$
for some $\rho>-1/\alpha$ and some $\widehat{\ell}$;
\item
if $p=1$, then $\lim_{t \to \infty} \int_0^t v(y) \dy = \infty$ and there exists
a positive monotone function $u$ such that $v(t) \sim u(t)$,
$t\to\infty$, or $v$ is directly Riemann integrable on $[0,\infty)$;
\item
if $p\in(0,1)$, then $X$ is independent of $\xi$.
\end{itemize}
Then, as $t\to\infty$,
\begin{equation}    \label{eq:convergence A3}
\frac{Y(ut)-\frac{1}{\mu}\int_0^{ut}h(y)\dy}{\sqrt{\int_0^t
v(y)\dy + c(t)^2h(t)^2}} ~\fdc~
\sqrt{\frac{(1-p)(1+\beta)}{\mu}}V_\beta(u)+
\sqrt{p}\mu^{-(\alpha+1)/\alpha}\int_0^u(u-y)^\rho \, {\rm d}
\mathcal{S}_\alpha(y),
\end{equation}
where $V_\beta$ is as in Definition \ref{definition_v_process},
and $\mathcal{S}_\alpha$ is assumed independent of $V_\beta$.
\end{thm}

\begin{thm} \label{Thm:mu=infty}
Suppose that \eqref{eq:domain alpha} holds for $\alpha\in(0,1)$
and that $h$ is not identically zero.
Assume further that the following limit
$$
q:=\lim_{t\to\infty}\frac{h(t)^2}{v(t)\mmp\{\xi>t\}+h(t)^2}\in[0,1]
$$
exists and that
\begin{itemize}
	\item
        if $q<1$, then the assumptions of Proposition \ref{Prop:mu=infty} hold
(with the same $\alpha$ as above);
  \item
        if $q=1$, then $h(t)\sim t^\rho\widehat{\ell}(t)$,
$t\to\infty$ for some $\rho\geq -\alpha$ and some
$\widehat{\ell}$; if $\rho=-\alpha$, then there exists a positive
increasing function $w$ such that $\lim_{t \to \infty} w(t)=\infty$ and $\lim_{t \to \infty} \frac{h(t)}{\mmp\{\xi>t\}w(t)}=1$.
\end{itemize}
Then, setting $\rho:=(\beta-\alpha)/2$ when $q\in(0,1)$,
\begin{equation*}
\frac{\mmp\{\xi>t\}Y(ut)}{\sqrt{v(t)\mmp\{\xi>t\}+h(t)^2}} ~\fdc~
\sqrt{1-q}Z_{\alpha,\beta}(u)+\sqrt{q}\int_{[0,u]}(u-y)^{\rho}
\, {\rm d}W^\leftarrow_\alpha(y), \quad t\to\infty,
\end{equation*}
where $Z_{\alpha,\beta}$ is as in Definition
\ref{definition_process_z}, and $W^\leftarrow_\alpha$ under the
integral sign is the same as in the definition of
$Z_{\alpha,\beta}$. In particular, the summands defining the
limit process are dependent.
\end{thm}

There is a simple situation where the weak convergence of
finite-dimensional distributions obtained in Theorem
\ref{Thm:mu=infty} implies the $J_1$-convergence on $D$. Of
course, the case where the limit process in Proposition
\ref{Prop:mu=infty} is a conditional white noise (equivalently,
$C(u,w)=0$ for $u\neq w$) must be eliminated as no version of such
a process belongs to $D$.
\begin{cor} \label{Cor:J_1 convergence}
Let $X(t)$ be almost surely increasing with $\lim_{t \to \infty} X(t)\in
(0,\infty]$ almost surely. Assume that the assumptions of Theorem
\ref{Thm:mu=infty} are in force with the exception that in the
case $q<1$ the conditions on the function $f(u,w)$ are replaced by
the condition that the function $(u,w)\mapsto \me [X(u)X(w)]$ is
regularly varying in $\mr_+^2$ of index $\beta$ with limit
function $C$. Then the limit relations of Theorem
\ref{Thm:mu=infty} hold in the sense of weak convergence in the
$J_1$-topology on $D$, where $Z_{\alpha,\beta}(0)=0$ is defined
as the limit in probability of $Z_{\alpha,\beta}(u)$ as $u\downarrow 0$.
\end{cor}
We close the section with a negative result which implies that
weak convergence of the finite-dimensional distributions in
Theorem \ref{Thm:mu=infty} cannot be strengthened to weak
convergence on $D(0,\infty)$ whenever $Z_{\alpha,\alpha}$ arises
in the limit.

\begin{assertion}   \label{main3}
Any version of the process $Z_{\alpha,\alpha}$ has paths in the
Skorokhod space $D(0,\infty)$ with probability strictly less than
$1$. If further $C(u,w)=0$ for all $u\neq w$, $u,w>0$, then any
version has paths in $D(0,\infty)$ with probability $0$.
\end{assertion}

\section{Applications}\label{sec:Applications}
Unless the contrary is stated, the random variable $\eta$
appearing in this section may be arbitrarily dependent on $\xi$,
and $(\xi_k,\eta_k)$, $k \in \mn$ denote i.i.d.~copies of
$(\xi,\eta)$.

\begin{example}
Let $X(t) = \1_{\{\eta> t\}}$, $\sigma^2<\infty$ and suppose that
$\mmp\{\eta>t\} \sim t^\beta \ell(t)$ for some $\beta\in (-1,0]$.
Since $h(t) = \me [X(t)] = \mmp\{\eta>t\}$ and
$v(t)=\mmp\{\eta>t\}\mmp\{\eta\leq t\}$ we infer $\lim_{t \to
\infty} v(t)/h(t)^2=\infty$. Further
\begin{equation*}
\frac{f(ut,wt)}{v(t)}=\frac{\mmp\{\eta>(u\vee w)t\} \mmp\{\eta
\leq (u\wedge w)t\}}{\mmp\{\eta>t\}\mmp\{\eta\leq t\}} ~\to~
(u\vee w)^\beta,  \quad u,w>0,
\end{equation*}
and this convergence is locally uniform in $\mr_+^2$ as it is the
case for $\lim_{t \to \infty} \mmp\{\eta>(u\vee
w)t\}/\mmp\{\eta>t\} = (u\vee w)^\beta$ by Lemma \ref{reg_var}(a).
In particular, condition \eqref{0040} holds with $C(u,w) = (u \vee
w)^\beta$. Finally, condition \eqref{eq:Lindeberg X-h} holds
because $|\1_{\{\eta>t\}}-\mmp\{\eta>t\}| \leq 1$ a.s. Now we
conclude that, according to the case $p=0$ of Theorem
\ref{Thm:mu<infty},
\begin{equation*}
\frac{\sum_{k\geq 0}\1_{\{S_k \leq ut <
S_k+\eta_{k+1}\}}-\frac{1}{\mu}\int_0^{ut}\mmp\{\eta>y\}\dy}{\sqrt{\mu^{-1}t\mmp\{\eta>t\}}}
~\fdc~ V_\beta(u),
\end{equation*}
where $V_\beta$ is a centered Gaussian process with covariance
\begin{equation*}
\me [V_\beta(u)V_\beta(w)] =
(1+\beta)^{-1}(w^{1+\beta}-(w-u)^{1+\beta}), \quad   0 \leq u \leq
w.
\end{equation*}
Assuming that $\xi$ and $\eta$ are independent, a counterpart of this result
with a random centering (i.e.~a result that follows from Proposition \ref{Prop:mu<infty})
was obtained in Proposition 3.2 of \citep{Mikosch+Resnick:2006}.
\end{example}

\begin{example}
Let $X(t)=\1_{\{\eta\leq t\}}$. Since $h(t)=\mmp\{\eta \leq t\}$
and $v(t)=\mmp\{\eta\leq t\}\mmp\{\eta>t\}\sim \mmp\{\eta>t\}$, we
infer $\lim_{t \to \infty} th(t)^2/\int_0^t v(y) \dy = \infty$.
Further, if $\me \eta<\infty$, then $v$ is dRi on $[0,\infty)$
because it is nonnegative, bounded, a.e.~continuous and dominated
by the decreasing and integrable function $\mmp\{\eta>t\}$. If
$\me\eta=\infty$, i.e., $\lim_{t \to \infty} \int_0^t
v(y)\dy=\infty$, $v$ is equivalent to the monotone function
$u(t)=\mmp\{\eta>t\}$. If $\sigma^2<\infty$ then, according to the
case $p=1$ of Theorem \ref{Thm:mu<infty},
\begin{equation*}
\frac{\sum_{k\geq 0}\1_{\{S_k+\eta_{k+1}\leq
ut\}}-\frac{1}{\mu}\int_0^{ut}\mmp\{\eta\leq
y\}\dy}{\sqrt{\sigma^2\mu^{-3}t}} ~\fdc~ \mathcal{S}_2(u),
\end{equation*}
where $\mathcal{S}_2$ is a Brownian motion, because $h$ is
regularly varying at $\infty$ of index $\rho=0$. If
$\mmp\{\xi>t\}$ is regularly varying at $\infty$ of index
$-\alpha$, $\alpha\in (0,1)$, then, by Corollary \ref{Cor:J_1
convergence},
\begin{equation*}
\mmp\{\xi>t\}\sum_{k\geq 0}\1_{\{S_k+\eta_{k+1}\leq ut\}}
~\Rightarrow~ W^\leftarrow_\alpha(u)
\end{equation*}
in the $J_1$-topology on $D$.
\end{example}


\begin{example}
Let $X(t)=\eta g(t)$ with $\Var\,\eta<\infty$ and let
$g:\mr^+\to\mr$ be regularly varying at $\infty$ of index
$\beta/2$ for some $\beta>-1$. Then $h(t) = g(t) \me\eta$ and
$v(t)= g(t)^2 \, \Var \, \eta$. While $f(u,w)= g(u)g(w) \Var \eta$ is clearly regularly varying in $\mr_+^2$ of index
$\beta$ with limit function $C(u,w)=(uw)^{\beta/2}$, \eqref{0040}
holds by virtue of Lemma \ref{reg_var}(a). Further observe that
$\lim_{t \to \infty} \sqrt{tv(t)}/|g(t)|=\infty$ implies
\begin{align*}
\me \big[(X(t)-h(t))^2 & \1_{\{|X(t)-h(t)|>y\sqrt{tv(t)}\}}\big] \\
 & = g(t)^2 \, \me\big[(\eta-\me\eta)^2\1_{\{|\eta-\me\eta|>y\sqrt{tv(t)}/|g(t)|\}}\big] = o(v(t))
\end{align*}
and thereupon \eqref{eq:Lindeberg X-h}. Also, as a consequence of
$\lim_{t \to \infty} \sqrt{v(t)/\mmp\{\xi>t\}}/|g(t)| = \infty$, which holds
whatever the law of $\xi$ is, we have
\begin{align*}
\me \big[(X(t)-h(t))^2 & \1_{\{|X(t)-h(t)|>y\sqrt{v(t)/\mmp\{\xi>t\}}\}}\big]   \\
&= g(t)^2 \, \me \big[(\eta-\me\eta)^2\1_{\{|\eta-\me\eta|>y
\sqrt{v(t)/\mmp\{\xi>t\}}/|g(t)|\}}\big] = o(v(t))
\end{align*}
which means that condition \eqref{eq:stable Lindeberg X-h} holds.

If $\me\eta=0$ and $\mu\in (0,\infty)$, then, according to
Proposition \ref{Prop:mu<infty},
\begin{equation*}
\frac{\sum_{k \geq 0}\eta_{k+1}g(ut-S_k)\1_{\{S_k\leq ut\}}}{
\sqrt{\mu^{-1}t\me [\eta^2]}g(t)} ~\fdc~ V_{\beta}(u)
\end{equation*}
where $V_\beta$ is a centered Gaussian process with covariance
\begin{equation*}
\me [V_\beta(u) V_\beta(w)] = \int_0^u
(u-y)^{\beta/2}(w-y)^{\beta/2} \, \dy,    \quad   0 < u \leq w.
\end{equation*}
Furthermore, the limit process can be represented as a stochastic
integral
\begin{equation*}V_\beta(u)=\int_{[0,u]}(u-y)^{\beta/2} \, {\rm
d}\mathcal{S}_2(y), \quad   u>0.
\end{equation*}

Throughout the rest of this example we assume that $\eta$ is
independent of $\xi$.

If $\me\eta=0$ and $\mmp\{\xi>t\}$ is regularly varying at
$\infty$ of index $-\alpha$, $\alpha\in (0,1)$ and $\beta >
-\alpha$ then, according to Proposition \ref{Prop:mu=infty},
\begin{equation*}
\frac{\sqrt{\mmp\{\xi>t\}}}{g(t)} \sum_{k \geq 0}
\eta_{k+1}g(ut-S_k) \1_{\{S_k\leq ut\}} ~\fdc~ \sqrt{\me[\eta^2]}
Z_{\alpha,\beta}(u).
\end{equation*}
Furthermore, the limit process can be represented as a stochastic
integral
\begin{equation*}
Z_{\alpha,\beta}(u) = \int_{[0,u]}(u-y)^{\beta/2} \, {\rm
d}\mathcal{S}_2 (W^\leftarrow_\alpha(y)), \quad   u>0
\end{equation*}
where $\mathcal{S}_2$ is a Brownian motion independent of
$W^\leftarrow_\alpha$, which can be seen by calculating the
conditional covariance of the last integral.

If $\me\eta\neq 0$, $\sigma^2<\infty$ and $g$ is eventually
monotone, then, according to Theorem \ref{Thm:mu<infty},
\begin{align*}
\frac{\sum_{k\geq 0}\eta_{k+1}g(ut-S_k)\1_{\{S_k \leq ut\}} -
\mu^{-1} \me \eta \int_0^{ut}g(y) \dy}{\me \eta \sqrt{t}g(t)}
 \qquad  \qquad  \qquad  \qquad  \\
~\fdc~ \Big(\frac{\sigma^2}{\mu^3}\Big)^2
\int_{[0,u]}(u-y)^{\beta/2} \, {\rm d}\mathcal{S}_2(u) +
\Big(\frac{\Var\, \eta}{(\me\eta)^2\mu}\Big)^{\!1/2}V_\beta(u).
\end{align*}

If $\me\eta\neq 0$, $\mmp\{\xi>t\}$ is regularly varying at
$\infty$ of index $-\alpha$, $\alpha\in (0,1)$, and
$\beta>-2\alpha$, then, since $\lim_{t \to \infty} v(t)
\mmp\{\xi>t\}/h(t)^2=0$, an application of Theorem
\ref{Thm:mu=infty} with $q=1$ gives
\begin{equation*}
\frac{\mmp\{\xi>t\}}{g(t)} \sum_{k\geq 0} \eta_{k+1}g(ut-S_k)
\1_{\{S_k\leq ut\}} ~\fdc~ \me \eta
\int_{[0,u]}(u-y)^{\beta/2}{\rm d}W_\alpha^\leftarrow(y).
\end{equation*}
If further $\eta\geq 0$ a.s.~and $g$ is increasing (which implies
$\beta\geq 0$), then, according to Corollary \ref{Cor:J_1
convergence}, the limit relation takes place in the $J_1$-
topology on $D$.
\end{example}

\begin{example} \label{Exa:stationary OU}
Let $Z:=(Z(t))_{t\geq 0}$ be a stationary Ornstein-Uhlenbeck
process defined by
\begin{equation*}
Z(t) = e^{-t} \theta + \int_{[0,t]} e^{-(t-y)} \, {\rm
d}\mathcal{S}_2(y), \quad   t \geq 0
\end{equation*}
where $\theta$ is a normal random variable with mean zero and
variance $1/2$ independent of a Brownian motion $\mathcal{S}_2$.
$Z$ and $\xi$ may be arbitrarily dependent. Put
$X(t)=(t+1)^{\beta/2}Z(t)$ for $\beta\in (-1,0)$. Then
$\me[X(t)]=0$ and
$f(u,w)=\me[X(u)X(w)] = 2^{-1} (u+1)^{\beta/2}(w+1)^{\beta/2}e^{-|u-w|}$   
from which we conclude that $f$ is fictitious regularly varying in
$\mr_+^2$ of index $\beta$. By stationarity, for each $t>0$,
$Z(t)$ has the same law as $\theta$. Hence
\begin{equation*}
\me [X(t)^2\1_{\{|X(t)|>y\}}] = (t+1)^\beta
\me[\theta^2\1_{\{|\theta|>y(t+1)^{-\beta/2}\}}] = o(t^\beta),
\end{equation*}
i.e., condition \eqref{eq:Lindeberg X-h} holds. If $\mu<\infty$ an
application of Proposition \ref{Prop:mu<infty} yields
\begin{equation*}
\frac{\sum_{k \geq 0} X_{k+1}(ut-S_k) \1_{\{S_k\leq
ut\}}}{\sqrt{(2\mu)^{-1}t^{\beta+1}}} ~\fdc~ V_\beta(u),
\end{equation*}
the limiting process being a centered Gaussian process with
independent values (white noise).
\end{example}

\begin{example}
Let $X(t)=\mathcal{S}_2((t+1)^{-\alpha})$, $\mmp\{\xi>t\}\sim
t^{-\alpha}$ and assume that $X$ and $\xi$ are independent. Then
$f(u,w) = \me[X(u)X(w)]$ is uniformly regularly varying of index
$-\alpha$ in strips in $\mr_+^2$ with limit function
$C(u,w)=(u\vee w)^{-\alpha}$. \eqref{eq:stable Lindeberg X-h}
follows from
$$\me [X(t)^2\1_{\{|X(t)|>y\}}] = (t+1)^{-\alpha} \me [\mathcal{S}_2(1)^2\1_{\{|\mathcal{S}_2(1)|>y(t+1)^{\alpha/2}\}}] = o(t^{-\alpha})$$
for all $y>0$. Thus, Proposition \ref{Prop:mu=infty} (in which we
take $u(t)\equiv 1$) applies and yields $\sum_{k \geq 0}
X_{k+1}(ut-S_k) \1_{\{S_k\leq ut\}} ~\fdc~
Z_{\alpha,\alpha}(u)$.
\end{example}

\section{Proofs of main results}    \label{sec:Proofs of the main results}
\subsection{Proofs of Propositions \ref{Prop:mu<infty} and \ref{Prop:mu=infty}}\label{subsec:Proofs of Propositions}
For a $\sigma$-algebra $\mathcal{G}$ we shall write
$\me_{\mathcal{G}}[\cdot]$ for $\me[\cdot|\mathcal{G}]$. Recalling
that $\nu(t)=\inf\{k\in\mn_0: S_k>t\}$, $t \geq 0$, we define the
renewal function $U(t) := \me [\nu(t)] = \sum_{k\geq 0}\mmp\{S_k \leq t\}$, $t\geq 0$.

\begin{proof}[Proof of Proposition \ref{Prop:mu<infty}]
We only investigate the case where $C(u,w)>0$ for some $u,w>0$,
$u\neq w$. Modifications needed in the case where $C(u,w)=0$ for
all $u,w>0$, $u\neq w$ should be clear from the subsequent
presentation.

Note that relation \eqref{0040} ensures continuity of the function $u \mapsto C(u,u+w)$ on $(0,\infty)$ for each $w>0$
(an accurate proof of a similar fact is given in \cite[pp.~2--3]{Yakimiv:2005}).
From the Cauchy-Schwarz inequality, we deduce that
\begin{equation}    \label{eq:impo ineq1}
|f(u,w)| \leq 2^{-1}(v(u)+v(w)), \quad u, w\geq 0,
\end{equation}
and hence
\begin{equation}    \label{eq:impo ineq2}
C(u-y,w-y) \leq 2^{-1} ((u-y)^\beta+(w-y)^\beta).
\end{equation}
Consequently, as $\beta > -1$,
\begin{equation*}
\int_0^u C(u-y,w-y) \, \dy<\infty,
\quad   0 < u \leq w.
\end{equation*}
Since $(u,w) \mapsto C(u,w)$ is positive semidefinite, so is
$(u,w)\mapsto \int_0^u C(u-y,w-y) \dy$, $0 < u \leq w$. Hence the
process $V_\beta$ does exist.

Without loss of generality we can and do assume that $X$ is
centered, for it is the case for $X(t)-h(t)$. According to the
Cram\'er-Wold device (see Theorem 29.4 in \citep{Billingsley:2012})
it suffices to prove that
\begin{equation}    \label{eq:Cramer-Wold device}
\frac{\sum_{j=1}^m \alpha_j \sum_{k \geq
0}X_{k+1}(u_jt-S_k)\1_{\{S_k \leq u_jt\}}}{\sqrt{\mu^{-1}tv(t)}}
~\stackrel{\mathrm{d}}{\to}~ \sum_{j=1}^m \alpha_j V_\beta(u_j)
\end{equation}
for all $m\in\mn$, all $\alpha_1,\ldots, \alpha_m\in\mr$ and all
$0<u_1<\ldots<u_m<\infty$. Note that the random variable
$\sum_{j=1}^m \alpha_j V_\beta(u_j)$ has a normal law with mean
$0$ and variance
\begin{equation}\label{D_definition}
(1+\beta)^{-1}\sum_{j=1}^m \alpha_j^2 u_j^{1+\beta}
+2\sum_{1\leq i<j\leq m}\alpha_i\alpha_j\int_0^{u_i}C(u_i-y,u_j-y) \, \dy ~=:~ D(u_1,\ldots, u_m).
\end{equation}

Define the $\sigma$-algebras $\F_0:=\{\varnothing,\Omega\}$ and
$\F_k:=\sigma((X_1,\xi_1),\ldots,(X_k,\xi_k))$, $k\in\mn$ and observe that
\begin{equation*}
\me_{\F_k} \bigg[\sum_{j=1}^m \alpha_j \1_{\{S_k \leq u_jt\}}
X_{k+1}(u_jt-S_k)\bigg] = 0.
\end{equation*}
Thus, in order to prove \eqref{eq:Cramer-Wold device}, one may use
the martingale central limit theorem  (Corollary 3.1 in
\citep{Hall+Heyde:1980}), whence it suffices to verify
\begin{align}       \label{eq:mgale CLT1}
\sum_{k\geq 0} \me_{\F_k} [Z_{k+1,t}^2] ~\stackrel{\mmp}{\to}~
D(u_1,\ldots, u_m),
\end{align}
and
\begin{equation}    \label{eq:mgale CLT2}
\sum_{k\geq 0} \me_{\F_k}
\big[Z_{k+1,t}^2\1_{\{|Z_{k+1,t}|>y\}}\big]
~\stackrel{\mmp}{\to}~ 0
\end{equation}
for all $y>0$, where
\begin{equation*}
Z_{k+1,t} := \frac{\sum_{j=1}^m \alpha_j \1_{\{S_k\leq
u_jt\}}X_{k+1}(u_jt-S_k)}{\sqrt{\mu^{-1}tv(t)}}, \quad k\in\mn_0,\
t>0.
\end{equation*}

\noindent {\it Proof of \eqref{eq:mgale CLT2}}: In view of the
inequality
\begin{eqnarray}\label{for_ref}
(a_1+\ldots+a_m)^2\1_{\{|a_1+\ldots+a_m|>y\}}&\leq&
(|a_1|+\ldots+|a_m|)^2\1_{\{|a_1|+\ldots+|a_m|>y\}}\notag\\&\leq&
m^2 (|a_1| \vee\ldots\vee |a_m|)^2\1_{\{m(|a_1| \vee\ldots\vee
|a_m|)>y\}}\notag\\&\leq&
m^2\big(a_1^2\1_{\{|a_1|>y/m\}}+\ldots+a_m^2\1_{\{|a_m|>y/m\}}\big)
\end{eqnarray}
which holds for $a_1,\ldots,a_m\in\mr$, it is sufficient to show that
\begin{equation}\label{Proof_of_Prop21_conv_to_0}
\sum_{k\geq 0} \! \1_{\{S_k \leq t\}} \me_{\F_k} \!
\bigg[\frac{X_{k+1}(t-S_k)^2}{\mu^{-1}tv(t)}
\!\1_{\{|X_{k+1}(t-S_k)|>y\sqrt{\mu^{-1}tv(t)}\}}\bigg]
~\stackrel{\mmp}{\to}~ 0
\end{equation}
for all $y>0$. We can take $t$ instead of $u_jt$ here because $v$
is regularly varying and $y>0$ is arbitrary.

Without loss of generality we assume that the function $t \mapsto tv(t)$ is increasing,
for we could otherwise work with $(\beta+1)\int_0^t v(y) {\rm d}y$ (see Lemma \ref{reg_var}(c)).
By Markov's inequality and the aforementioned monotonicity relation
\eqref{Proof_of_Prop21_conv_to_0} follows if we can prove that
\begin{equation}    \label{eq:renewal asymptotics}
\lim_{t \to \infty} \frac{1}{tv(t)}\int_{[0,t]} v_y(t-x) \, {\rm d} U(x) = 0
\end{equation}
for all $y>0$, where the definition of $v_y$ is given in \eqref{eq:Lindeberg X-h}.
Recalling that $\mu<\infty$ and that $v$ is locally
bounded, measurable and regularly varying at infinity of index
$\beta\in (-1,\infty)$ an application of Lemma \ref{Lem:Sgibnev2}
with $r_1=0$ and $r_2=1$ yields
\begin{equation*}
\int_{[0,t]} v(t-x) \, {\rm d}U(x)  ~\sim~  {\rm const}\,tv(t).
\end{equation*}
Since, according to \eqref{eq:Lindeberg X-h}, $v_y(t)=o(v(t))$,
\eqref{eq:renewal asymptotics} follows from Lemma
\ref{equiv:Lemma}(b).

\noindent
{\it Proof of \eqref{eq:mgale CLT1}}:
It can be checked that
\begin{eqnarray*}
\sum_{k\geq 0} \me_{\F_k} [Z_{k+1,t}^2]
& = &
\frac{\sum_{j=1}^m\alpha_j^2 \sum_{k \geq 0}\1_{\{S_k\leq u_jt\}}v(u_jt-S_k)}{\mu^{-1} t v(t)}  \\
& & + \frac{2\sum_{1\leq i<j\leq m}\alpha_i\alpha_j \sum_{k \geq 0}\1_{\{S_k\leq u_it\}}f(u_it-S_k, u_jt-S_k)}{\mu^{-1}tv(t)}.
\end{eqnarray*}
We shall prove that
\begin{equation}    \label{eq:convergence at u_1}
\frac{\sum_{k\geq 0} \1_{\{S_k\leq u_it\}}v(u_it-S_k)}{\mu^{-1}tv(t)}
~=~ \frac{\int_{[0,u_i]} v((u_i-y)t) \, {\rm d}\nu(ty)}{\mu^{-1} t v(t)}
~\stackrel{\mmp}{\to}~ \frac{u_i^{1+\beta}}{1+\beta}
\end{equation}
and
\begin{eqnarray}    \label{eq:convergence at u_1,u_2}
\frac{\sum_{k\geq 0}\1_{\{S_k\leq
u_it\}}f(u_it-S_k,u_jt-S_k)}{\mu^{-1}tv(t)} & = &
\frac{\int_{[0,u_i]}f((u_i-y)t,(u_j-y)t) \, {\rm d}\nu(ty)}{\mu^{-1} t v(t)}    \notag  \\
& \stackrel{\mmp}{\to} &
\int_0^{u_i} C(u_i-y,u_j-y) \dy.
\end{eqnarray}
for all $1\leq i<j\leq m$.

Fix any $u_i<u_j$ and pick $\varepsilon \in (0,u_i)$.
By the functional strong law of large numbers (Theorem 4 in \citep{Glynn+Whitt:1988})
\begin{equation*}
\lim_{t \to \infty}
\sup_{y\in[0,u_i]}\Big|\frac{\nu(ty)}{\mu^{-1}t}-y\Big|=0	\quad	\text{a.s.}
\end{equation*}
Also,
\begin{equation*}
\lim_{t \to \infty} \frac{v((u_i-y)t)}{v(t)} = (u_i-y)^\beta
\end{equation*}
uniformly in $y \in [0,u_i-\varepsilon]$ by Lemma
\ref{reg_var}(a), and
\begin{equation*}
\lim_{t \to \infty} \frac{f((u_i-y)t,(u_j-y)t)}{v(t)} = C(u_i-y,u_j-y)
\end{equation*}
uniformly in $y\in [0,u_i-\varepsilon]$, by virtue of
\eqref{0040}. Two applications of Lemma \ref{Lem:continuous
mapping D}(a) (with $X_t(y)=\nu(ty)/(\mu^{-1}t)$) yield
\begin{equation*}
\int_{[0,u_i-\varepsilon]}\frac{v((u_i-y)t)}{v(t)} \, {\rm d}{\nu(ty)\over \mu^{-1}t}
~\overset{\mathbb{P}}{\to}~
\int_0^{u_i-\varepsilon}(u_i-y)^\beta \, \dy    ~=~
\frac{u_i^{1+\beta}-\varepsilon^{1+\beta}}{1+\beta}
\end{equation*}
and
\begin{equation*}
\int_{[0,u_i-\varepsilon]} \frac{f((u_i-y)t,(u_j-y)t)}{v(t)} \, {\rm d} \frac{\nu(ty)}{\mu^{-1} t}
~\overset{\mmp}{\to}~
\int_0^{u_i-\varepsilon}C(u_i-y, u_j-y) \, \dy.
\end{equation*}
Observe that since $\nu(y)$ is a.s.\ increasing, so is $X_t(y)$.

As $\varepsilon\downarrow 0$, the right-hand sides of the last two
equalities converge to $(1+\beta)^{-1}u_i^{1+\beta}$ and
$\int_0^{u_1} C(u_i-y,u_j-y) \dy$, respectively. Therefore, for
\eqref{eq:convergence at u_1} and \eqref{eq:convergence at
u_1,u_2} to hold it is sufficient (see Lemma \ref{Lem:Thm 4.2 of
Billingsley}) that
\begin{equation*}
\lim_{\varepsilon\downarrow 0} \limsup_{t\to\infty} \mmp\bigg\{
\frac{\int_{(u_i-\varepsilon, \, u_i]} v(t(u_i-y)) \, {\rm
d}\nu(ty)}{tv(t)} > \delta\bigg\} = 0
\end{equation*}
and
\begin{equation*}
\lim_{\varepsilon \downarrow 0} \limsup_{t\to\infty} \mmp\bigg\{
\frac{\big|\int_{(u_i-\varepsilon, \, u_i]}f(t(u_i-y),t(u_j-y)) \,
{\rm d}\nu(ty) \big|}{tv(t)}>\delta\bigg\}=0
\end{equation*}
for all $\delta>0$.
By Markov's inequality it thus suffices to check that
\begin{equation}    \label{eq:Markov's inequality at u_1}
\lim_{\varepsilon\downarrow 0} \limsup_{t\to\infty}
\frac{\int_{(u_i-\varepsilon, \, u_i]}v((u_i-y)t) \, {\rm d}U(ty)}{tv(t)} ~=~ 0
\end{equation}
and
\begin{equation}    \label{eq:Markov's inequality at u_1,u_2}
\lim_{\varepsilon\downarrow 0} \limsup_{t\to\infty}
\frac{\int_{(u_i-\varepsilon, \, u_i]} |f((u_i-y)t,(u_j-y)t)| \, {\rm d}U(ty)}{tv(t)} ~=~ 0,
\end{equation}
respectively. Changing the variable $s=u_it$ and recalling that
$v$ is regularly varying of index $\beta\in (-1,\infty)$ we apply
Lemma \ref{Lem:Sgibnev2} with $r_1=1-\varepsilon u_i^{-1}$ and
$r_2=1$ to infer
\begin{eqnarray*}
\int_{((u_i-\varepsilon)t, \, u_it]} v(u_it-y) \, {\rm d}U(y)
&=& \int_{((1-\varepsilon u_i^{-1})s,s]}v(s-y)\,{\rm d}U(y)\\
&\sim& \bigg(\frac{\varepsilon}{u_i}\bigg)^{1+\beta}\frac{sv(s)}{(1+\beta)\mu}
~\sim~ \frac{\varepsilon^{1+\beta} tv(t)}{(1+\beta)\mu}.
\end{eqnarray*}
Using \eqref{eq:impo ineq1} 
we further obtain
\begin{align*}
\int_{((u_i-\varepsilon)t, \, u_it]} & |f(u_it-y, u_jt-y)| \, {\rm d}U(y)   \\
& \leq~
\frac{1}{2} \int_{((u_i-\varepsilon)t, \, u_it]} v(u_it-y) \, {\rm d}U(y)+ {1\over 2}
\int_{((u_i-\varepsilon)t, \, u_it]} v(u_jt-y) \, {\rm d}U(y)    \\
& \sim~ \frac{1}{2\mu(1+\beta)}
\big(\varepsilon^{1+\beta}+(u_j-u_i+\varepsilon)^{1+\beta}-(u_j-u_i)^{1+\beta}\big)tv(t),
\end{align*}
where for the second integral we have changed the variable
$s=u_jt$, invoked Lemma \ref{Lem:Sgibnev2} with
$r_1=(u_i-\varepsilon)u_j^{-1}$ and $r_2=u_iu_j^{-1}$ and then got
back to the original variable $t$. These relations entail both,
\eqref{eq:Markov's inequality at u_1} and \eqref{eq:Markov's
inequality at u_1,u_2}. The proof of Proposition
\ref{Prop:mu<infty} is complete.
\end{proof}

In what follows, $\F$ denotes the $\sigma$-algebra generated by
$(S_n)_{n\in\mn_0}$.
\begin{proof}[Proof of Proposition \ref{Prop:mu=infty}]
As in the previous proof we can and do assume that $X$ is
centered. Put $r(t):=v(t)/\mmp\{\xi>t\}$. The process
$Z_{\alpha,\beta}$ is well-defined by Lemma \ref{integrals}. In
view of the Cram\'er-Wold device it suffices to check that
\begin{equation}\label{0009}
\frac{1}{\sqrt{r(t)}}\sum_{j=1}^m \gamma_j Y(u_jt) \ \dod \
\sum_{j=1}^m \gamma_j Z_{\alpha,\beta}(u_j)
\end{equation}
for all $\gamma_1,\ldots, \gamma_m\in\mr$. Since $C(y,y)=y^\beta$,
then, given $W^\leftarrow_\alpha$, the random variable
$\sum_{j=1}^m \gamma_j Z_{\alpha,\beta}(u_j)$ is centered normal
with variance
\begin{eqnarray}\label{xxx}
D_{\alpha,\beta}(u_1,\ldots, u_m)
&:=&	\sum_{j=1}^m \gamma_j^2 \int_{[0,u_j]}(u_j-y)^\beta \, {\rm d} W^\leftarrow_\alpha(y)\\
& &		+2\sum_{1\leq i<j\leq m}\gamma_i\gamma_j \int_{[0,u_i]}C(u_i-y,u_j-y) \, {\rm d}W^\leftarrow_\alpha(y).\notag
\end{eqnarray}
Equivalently,
\begin{equation*}
\me \bigg[\exp\bigg({\rm i} z \sum_{j=1}^m \gamma_j
Z_{\alpha,\beta}(u_j)\bigg)\bigg] = \me
\big[\exp(-D_{\alpha,\beta}(u_1,\ldots, u_m)z^2/2)\big], \ \
z\in\mr
\end{equation*}
where here and throughout the paper, ${\rm i}$ denotes the imaginary unit.
Hence, according to Lemma \ref{essential2}, \eqref{0009} is a consequence of
\begin{equation}\label{00100}
\sum_{k\geq 0} \me_\mf [Z_{k+1,t}^2] \ \dod \
D_{\alpha,\beta}(u_1,\ldots, u_m),
\end{equation}
where $Z_{k+1,t} := (r(t))^{-1/2}\sum_{j=1}^m \gamma_j
X_{k+1}(u_jt-S_k)\1_{\{S_k\leq u_jt\}}$, and
\begin{equation}\label{00080}
\sum_{k\geq 0} \me_\mf [Z_{k+1,t}^2\1_{\{|Z_{k+1,t}|>y\}}] \
\tp\ 0
\end{equation}
for all $y>0$. Since $r(t)$ is regularly varying at $\infty$ of
index $\beta+\alpha$ we have
\begin{eqnarray*}
&&\hspace{-3cm}\limsup_{t\to\infty} {1\over r(t)} \int_{(\rho z,z]}v(t(z-y)) \,
{\rm d}U(ty)\\
&&\leq \lim_{t \to \infty} {r(tz)\over r(t)}\limsup_{t\to\infty}
{1\over r(tz)} \int_{(\rho tz,tz]}v(tz-y) \, {\rm
d}U(y)\\
&&=z^{\beta+\alpha}\limsup_{t\to\infty} {1\over r(t)}
\int_{(\rho t,t]}v(t-y) \, {\rm d}U(y)
\end{eqnarray*}
for all $z>0$. Hence the relation
\begin{equation}\label{former_essential_1}
\lim_{\rho\uparrow 1} \limsup_{t\to\infty} {1\over r(t)}
\int_{(\rho z,z]}v(t(z-y)) \, {\rm d}U(ty) ~=~ 0
\end{equation}
for all $z>0$ is an immediate consequence of Lemma
\ref{Lem:convergence of 1st moment at t=1}(a). Using the
representation
\begin{align*}
\sum_{k\geq 0} \me_\mf [Z_{k+1,t}^2]
& = \frac{1}{r(t)}\int_{[0,u_m]}\bigg(\sum_{j=1}^m \gamma^2_j v((u_j-y)t)\1_{[0,u_j]}(y)		\\
& \hphantom{= \frac{1}{r(t)}\int_{[0,u_m]}\bigg(}
+2\sum_{1\leq i<j\leq m}\gamma_i\gamma_j
f((u_i-y)t,(u_j-y)t)\1_{[0,u_i]}(y)\bigg) \, {\rm d}\nu(ty)
\end{align*}
we further conclude that \eqref{00100} follows from Lemma \ref{principal}
with $\lambda_1=0$ (observe that conditions \eqref{additional} and
\eqref{ref3} are then not needed and \eqref{ref1} coincides with
\eqref{former_essential_1}). In view of \eqref{for_ref},
\eqref{00080} is a consequence of
\begin{equation}\label{inter1}
{1\over r(t)} \sum_{k\geq 0}\1_{\{S_k\leq t\}}
\me_{\F}\Big[(X_{k+1}(t-S_k))^2
\1_{\{|X_{k+1}(t-S_k)|>y\sqrt{r(t)}\}}\Big] ~\stackrel{\mmp}{\to}~
0
\end{equation}
for all $y>0$. To prove \eqref{inter1} we assume, without loss of
generality, that the function $r$ is increasing, for in the case
$\beta=-\alpha$ it is asymptotically equivalent to an increasing
function $u(t)$ by assumption, while in the case $\beta>-\alpha$
the existence of such a function is guaranteed by Lemma
\ref{reg_var}(b) because $r$ is then regularly varying of positive
index. Using this monotonicity and recalling that we are assuming
that $h\equiv 0$, whence $v_y(t)=\me
\big[(X(t))^2\1_{\{|X(t)|>y\sqrt{r(t)}\}}\big]$, we conclude that
it is sufficient to check that
\begin{align*}
\me \bigg[\sum_{k\geq 0} \1_{\{S_k\leq t\}}
\me_{\F}\Big[(X_{k+1}(t-S_k))^2\1_{\{|X_{k+1}(t-S_k)|>y\sqrt{r(t-S_k)}\}}\Big] \bigg]  \\
~=~ \int_{[0,t]}v_y(t-x) \, {\rm d}U(x) ~=~ o(r(t))
\end{align*}
for all $y>0$, by Markov's inequality. In view of \eqref{eq:stable
Lindeberg X-h} the latter is an immediate consequence of Lemma
\ref{Lem:convergence of 1st moment at t=1}(b) with
$\phi_1(t)=v_y(t)$, $\phi(t)=v(t)$, $q(t)=u(t)$ and
$\gamma=\beta$. The proof of Proposition \ref{Prop:mu=infty} is
complete.
\end{proof}

\subsection{Proofs of Theorems \ref{Thm:mu<infty} and \ref{Thm:mu=infty}}

For the proof of Theorem \ref{Thm:mu<infty} we need two auxiliary results, Lemma \ref{Lem:convergence to 0} and Lemma \ref{main4}.
Replacing the denominator in \eqref{eq:1st summand convergence} by a function which grows faster
leads to weak convergence of finite-dimensional distributions to zero.
However, this result holds without the regular variation assumptions of Proposition \ref{Prop:mu<infty}.

\begin{lemma}   \label{Lem:convergence to 0}
Assume that
\begin{itemize}
    \item
        $\mu=\me \xi < \infty$;
    \item
        either
        $$\lim_{t \to \infty}\int_0^t v(y)\dy=\infty	\quad	\text{and}	\quad	\lim_{t \to \infty}\frac{v(t)}{\int_0^t v(y)\dy}=0$$
        and there exists a monotone function $u$ such that $v(t)\sim u(t)$ as $t \to \infty$, or $v$ is directly Riemann integrable (dRi)
        on $[0,\infty)$.
\end{itemize}
Then
\begin{equation}    \label{eq:convergence to 0}
\frac{Y(ut)-\sum_{k\geq 0}h(ut-S_k)\1_{\{S_k\leq ut\}}}{s(t)} ~\fdc~ 0,  \quad   t \to \infty
\end{equation}
for any positive function $s(t)$ regularly varying at $\infty$
which satisfies 
$$
\lim_{t \to \infty} s(t)^2/\int_0^t v(y)\dy=\infty.
$$
\end{lemma}
\begin{proof}
By Chebyshev's inequality and the Cram\'{e}r-Wold device, it suffices to prove that
\begin{equation*}
s(t)^{-2} \, \me \bigg[\bigg(Y(t)-\sum_{k\geq 0}h(t-S_k)\1_{\{S_k\leq t\}}\bigg)^{\!\!2} \bigg] ~\to~ 0.
\end{equation*}
The expectation above equals
$\int_{[0,t]}v(t-y){\rm d}U(y)$. If $v$ is dRi, the latter
integral is bounded (this is clear from the key renewal theorem
when the law of $\xi$ is nonlattice while in the lattice case, it
follows from Lemma 8.2 in \citep{Iksanov+Marynych+Vatutin:2013+}).
If $v$ is non-integrable and $u$ is a monotone function such that
$v(t)\sim u(t)$, Lemma \ref{equiv:Lemma}(a) with $r_1=0$ and
$r_2=1$ yields
$$
\int_{[0,t]}v(t-y) \, {\rm d}U(y) ~\sim~ \int_{[0,t]}u(t-y) \, {\rm d}U(y).
$$
Modifying $u$ if needed in the right vicinity of zero we can
assume that $u$ is monotone and locally integrable.
Since $u \sim v$, we have $\lim_{t \to \infty} (u(t)/\int_0^t u(y){\rm d}y)=0$ as the corresponding relation holds for $v$,
and an application of Lemma \ref{Lem:Sgibnev} applied to $\phi=u$ with
$r_1=0$ and $r_2=1$ gives
$$
\int_{[0,t]}u(t-y) \, {\rm d}U(y)~\sim~ \frac{1}{\mu}\int_0^t u(y) \, {\rm d}y
$$
and again using $u \sim v$ we obtain
\begin{equation*}
\int_0^t u(y) \, \dy ~\sim~ \int_0^t v(y) \, \dy ~=~ o(s(t)^2),
\end{equation*}
where the last equality follows from the assumption on $s$. The
proof of Lemma \ref{Lem:convergence to 0} is complete.
\end{proof}

\begin{lemma}   \label{main4}
Assume that $h$ is eventually monotone and eventually nonnegative
and that the law of $\xi$ belongs to the domain of attraction of
an $\alpha$-stable law, $\alpha\in (1,2]$ (i.e., relation
\eqref{eq:FLT for N(t)} holds). Then
\begin{equation*}
\frac{\sum_{k\geq 0} h(ut-S_k) \1_{\{S_k\leq ut\}} - \frac{1}{\mu} \int_0^{ut}h(y) \, \dy}{r(t)} ~\fdc~ 0,
\quad   t\to\infty
\end{equation*}
for any positive function $r(t)$ regularly varying at $\infty$ of
positive index satisfying 
$$
\lim_{t \to \infty}
\frac{r(t)}{c(t)h(t)}=\infty,
$$
where $c$ is the same as in
\eqref{eq:FLT for N(t)}.
\end{lemma}
\begin{proof}
Using the Cram\'er-Wold device and taking into account the regular variation of $r$,
it suffices to prove that
\begin{equation}    \label{eq:stronger norming}
\frac{\int_{[0,t]}h(t-y) \, {\rm d}(\nu(y)-\frac{y}{\mu})}{r(t)}
~=~	\frac{\sum_{k\geq 0}h(t-S_k)\1_{\{S_k\leq t\}} - \frac{1}{\mu} \int_0^t h(y) \, \dy}{r(t)}
~\stackrel{\mmp}{\to}~	0.
\end{equation}
By assumption, there exists a $t_0>0$ such that $h$ is monotone
and nonnegative on $[t_0,\infty)$. Let $h^\ast \in D$ be an
arbitrary function which coincides with $h$ on $[t_0,\infty)$.
Then, for $t>t_0$,
\begin{eqnarray*}
\bigg|\int_{[0,t]}\big(h(t-y)-h^\ast(t-y)\big) \, {\rm
d}\nu(y)\bigg| & = &
\bigg|\int_{(t-t_0,t]}\big(h(t-y)-h^\ast(t-y)\big) \, {\rm d}\nu(y)\bigg|  \\
& \leq &
\sup_{0 \leq y \leq t_0} \big|h(y)-h^\ast(y)\big| \, (\nu(t)-\nu(t-t_0))   \\
& \stackrel{\mathrm{d}}{\leq} & \sup_{0 \leq y \leq t_0}
\big|h(y)-h^\ast (y)\big| \, \nu(t_0),
\end{eqnarray*}
where $Z_1 \stackrel{\mathrm{d}}{\leq} Z_2$ means that
$\mmp\{Z_1>x\}\leq \mmp\{Z_2>x\}$ for all $x\in\mr$, and the last
inequality in the displayed formula follows from the
distributional subadditivity of $\nu$. Analogously,
\begin{equation*}
\bigg|\int_{[0,t]}\big(h(t-y)-h^\ast(t-y)\big) \, \dy \bigg|
~\leq~	\sup_{0 \leq y \leq t_0} \big|h(y)-h^\ast(y)\big| t_0.
\end{equation*}
Hence while proving \eqref{eq:stronger norming} we can replace $h$
with $h^\ast$. Choosing $t_0$ large enough we make $h^\ast$
monotone and nonnegative on $[0,\infty)$. Furthermore, if $h^\ast$
is increasing on $[t_0,\infty)$ we set $h^\ast(t)=0$ for $t\in
[0,t_0)$ thereby ensuring that $h^\ast(0)=0$.

\noindent{\em Case where $h^\ast$ is increasing}.
Integration by parts reveals that it is enough to prove
\begin{equation}    \label{eq:stronger norming int by parts}
{1\over r(t)} \int_{[0,1]}
\big(\nu(t)-\nu(t(1-y)-)-\mu^{-1}ty\big) \, {\rm
d}(-h^\ast(t(1-y))) ~\stackrel{\mmp}{\to}~ 0.
\end{equation}
By monotonicity, $h^\ast(t(1-y))/h^\ast(t)\leq 1$ for all $y\in
[0,1]$. Hence $\lim_{t \to \infty} {h^\ast(t(1-y))\over
r(t)/c(t)}=0$. For sufficiently large $t$, define finite measures
$\rho_t$ on $[0,1]$ by
\begin{equation*}
\rho_t([0,a]) = \frac{r(t)/c(t)-h^\ast(t(1-a))}{r(t)/c(t)}, \quad
a\in [0,1].
\end{equation*}
Then the $\rho_t$ converge weakly to $\delta_0$ as $t \to \infty$.
Applying the continuous mapping $\mathcal{V}:D\to D[0,1]$ with
$\mathcal{V}(f(\cdot))=f(1)-f((1-\cdot)-)$ to \eqref{eq:FLT for
N(t)} we obtain
$$
\frac{\nu(t)-\nu(t(1-y)-)-\mu^{-1}ty}{\mu^{-1-1/\alpha}c(t)}\quad
\Rightarrow\quad
\mathcal{S}_{\alpha}(1)-\mathcal{S}_{\alpha}((1-y)-)$$ in the
$J_1$- or $M_1$-topology on $D[0,1]$. Invoking Lemma
\ref{Lem:continuous mapping D}(b) yields \eqref{eq:stronger
norming int by parts}, since
$(\mathcal{S}_{\alpha}(1)-\mathcal{S}_{\alpha}((1-y)-))_{y\in[0,1]}$
is a.s. continuous at zero and
$\mathcal{S}_{\alpha}(1)-\mathcal{S}_{\alpha}(1-)=0$ a.s.

\noindent{\em Case where $h^\ast$ is decreasing}.
Integration by parts reveals that we have to prove
\begin{equation}    \label{inter2}
\frac{\nu(t)-\mu^{-1}t}{r(t)} h^\ast(t) \stackrel{\mmp}{\to} 0
\quad \text{and}  \quad {1\over r(t)} \int_{[0,t]}\!\!
\Big(\!\nu(t)-\nu((t-y)-)-\frac{y}{\mu}\Big) \, {\rm
d}(\!-h^\ast(y)) \stackrel{\mmp}{\to} 0.
\end{equation}
The first of these is a consequence of the assumption $\lim_{t \to \infty} r(t)/(c(t)h(t))=\infty$ and \eqref{eq:FLT for N(t)}.
Arguing as in the proof of Theorem 2.7 on pp.~2160-2161 in \citep{Iksanov+Marynych+Meiners:2014}
(note that $g(t)$ in \citep{Iksanov+Marynych+Meiners:2014} corresponds to
$\mu^{-1-1/\alpha}c(t)$ in this paper), we observe that the second relation in \eqref{inter2}
follows once we can check that
\begin{equation*}
\lim_{t \to \infty} \frac{\int_{[t_0,t]}y^{1/\alpha-\delta}{\rm d}(\!-h^\ast(y))}{t^{1/\alpha-\delta}r(t)/c(t)} ~=~ 0
\end{equation*}
for some $\delta\in (0,1/\alpha)$ and $t_0=t_0(\delta)>0$
specified in Lemma 3.2 of \citep{Iksanov+Marynych+Meiners:2014}
(recall that $\alpha=2$ corresponds to the case where the limit process in
\eqref{eq:FLT for N(t)} is a Brownian motion). 
Pick $\delta$ to further satisfy $\delta<\gamma$, where $\gamma$
is the index of regular variation of $r$. By Lemma \ref{ct} $c(t)$
is regularly varying at $\infty$ of index $1/\alpha$. Hence the
function $t\mapsto t^{1/\alpha-\delta}r(t)/c(t)$ is regularly
varying at $\infty$ of the positive index $\gamma-\delta$ which
particularly implies $\lim_{t \to \infty} t^{1/\alpha-\delta}r(t)/c(t)=\infty$.
Integration by parts yields
\begin{eqnarray*}
\frac{\int_{[t_0,t]}y^{1/\alpha-\delta}{\rm
d}\big(\!\!-\!h^\ast(y)\big)}{t^{1/\alpha-\delta}r(t)/c(t)} &=&
\frac{-t^{1/\alpha-\delta}h^\ast(t)}{t^{1/\alpha-\delta}r(t)/c(t)}\\
&+&\frac{t_0^{1/\alpha-\delta}h^\ast(t_0)}{t^{1/\alpha-\delta}r(t)/c(t)}
+\Big(\frac{1}{\alpha}-\delta\Big)\frac{\int_{t_0}^t
y^{1/\alpha-\delta-1}h^\ast(y)\dy}{t^{1/\alpha-\delta}r(t)/c(t)}.
\end{eqnarray*}
As $t\to\infty$, the first two terms converge to zero. As for the
third, observe that for any $d>0$ there exists $t(d)$ such that
$h^\ast(t) \leq d^{-1}r(t)/c(t)=d^{-1}t^{\gamma-1/\alpha}\ell(t)$
for all $t \geq t(d)$. With this at hand, we infer
\begin{eqnarray*}
\frac{\int_{t_0}^t
y^{1/\alpha-\delta-1}h^\ast(y)\dy}{t^{1/\alpha-\delta}r(t)/c(t)} &
= &
\frac{\int_{t_0}^{t(d)}y^{1/\alpha-\delta-1}h^\ast(y)\dy}{t^{1/\alpha-\delta}r(t)/c(t)}
+ \frac{\int_{t(d)}^t y^{1/\alpha-\delta-1}h^\ast(y)\dy}{t^{1/\alpha-\delta}r(t)/c(t)}   \\
& \leq & o(1)+ d^{-1}\int_{t(d)}^t y^{\gamma-\delta-1}\ell(y)\dy /
t^{\gamma-\delta}\ell(t) ~\to~ d^{-1}(\gamma-\delta +1)^{-1}
\end{eqnarray*}
by Lemma \ref{reg_var}(c). Letting $d \to \infty$ completes the
proof of Lemma \ref{main4}.
\end{proof}

\begin{proof}[Proof of Theorem \ref{Thm:mu<infty}]


\noindent{\it Case} $p=0$: According to Proposition \ref{Prop:mu<infty},
\eqref{eq:1st summand convergence} holds which is equivalent to
\begin{equation}    \label{eq:LT for martingale term}
{Y(ut)-\sum_{k\geq 0}h(ut-S_k)\1_{\{S_k\leq ut\}}\over
\sqrt{\int_0^t v(y){\rm d}y}}\quad
\overset{\mathrm{f.d.}}{\Rightarrow}\quad \sqrt{{1+\beta \over
\mu}}V_\beta(u)
\end{equation}
because $v$ is regularly varying at $\infty$ of index $\beta\in (-1,\infty)$.

Since $\big(\int_0^t v(y){\rm d}y\big)^{1/2}$ is regularly varying at
$\infty$ of positive index $\frac12(1+\beta)$ and
$$	\lit \frac{\sqrt{\int_0^t v(y){\rm d}y}}{c(t)|h(t)|}=+\infty,	$$
Lemma \ref{main4}\footnote{Lemma \ref{main4} requires that $h$ be
eventually monotone and eventually nonnegative. If $h$ is
eventually nonpositive we simply replace it with $-h$.} (with
$r(t)=\sqrt{\int_0^t v(y){\rm d}y}$) applies and yields
\begin{equation*}
{\sum_{k\geq 0}h(ut-S_k)\1_{\{S_k\leq
ut\}}-\mu^{-1}\int_0^{ut}h(y)\dy\over \sqrt{\int_0^t v(y){\rm
d}y}} ~\overset{\mathrm{f.d.}}{\Rightarrow}~  0.
\end{equation*}
Summing the last relation and \eqref{eq:LT for martingale term}
finishes the proof for this case because
$$\int_0^t v(y){\rm d}y	~\sim~	\int_0^t v(y){\rm d}y+c(t)^2h(t)^2.	$$

\noindent {\it Case} $p>0$:
Using Theorem 1.1 in \citep{Iksanov:2013}
when $h(t)$ is eventually nondecreasing
and Theorem 2.7 in \citep{Iksanov+Marynych+Meiners:2014}
when $h(t)$ is eventually nonincreasing
we infer
\begin{equation}    \label{eq:LT for SN with deterministic response}
\frac{\sum_{k\geq 0}h(ut-S_k)\1_{\{S_k\leq
ut\}}-\mu^{-1}\int_0^{ut}h(y) \dy}{c(t)h(t)}
~\overset{\mathrm{f.d.}}{\Rightarrow}~
\mu^{-(\alpha+1)/\alpha}\int_{[0,u]}(u-y)^\rho \, {\rm d}
\mathcal{S}_\alpha(y).
\end{equation}

\noindent{\it Subcase} $p=1$: By Lemma \ref{ct} $c(t)$ is
regularly varying at $\infty$ of index $1/\alpha$. Hence
$c(t)h(t)$ is regularly varying of positive index. If $v$ is dRi,
an application of Lemma \ref{Lem:convergence to 0} (with
$s(t)=c(t)h(t)$) yields
\begin{equation}\label{eq:Y(ut) centered convergence to 0}
\frac{Y(ut)-\sum_{k\geq 0}h(ut-S_k)\1_{\{S_k\leq ut\}}}{c(t)h(t)}
~\overset{\mathrm{f.d.}}{\Rightarrow}~ 0.
\end{equation}
If $\lim_{t\to \infty}\int_0^t v(y){\rm d}y=\infty$, then the assumption $\lim_{t \to \infty}
(c(t)^2h(t)^2/\int_0^t v(y) \dy) = \infty$ implies that $\lim_{t \to \infty}
(v(t)/\int_0^t v(y)\dy) = 0$.
{
To see this, we can assume without loss of generality that $v$ is monotone.
If $v$ is decreasing, then the claimed convergence follows immediately.
Hence, consider the case where $v$ is increasing.
Since $c(t)^2h(t)^2$ is regularly varying and $\int_0^t v(y)dy \geq v(t/2)t/2$,
we conclude that there exists an $a>0$ such that $\lim_{t \to \infty} t^a/v(t)=\infty$.
Let $a_*$ denote the infimum of these $a$.
Then, there exists $\varepsilon > 0$ such that
$t^{a_*+\varepsilon}/v(t)\to\infty$ whereas
$t^{a_*+\varepsilon-1}/v(t)\to 0$.
Consequently,
\begin{equation*}
\frac{v(t)}{\int_0^t v(y)dy}
~\leq~	\frac{v(t)}{\int_{t/2}^t v(y)dy}
~\leq~	\frac{2v(t)}{tv(t/2)}
~=~	2^{a_*+\varepsilon} \frac{v(t)}{t^{a_*+\varepsilon}} \frac{(t/2)^{a_*+\varepsilon-1}}{v(t/2)}
~\to~	0
\end{equation*}
because both factors tend to
zero by our choice of gamma (and varepsilion).
}
Invoking Lemma \ref{Lem:convergence
to 0} again allows us to conclude that \eqref{eq:Y(ut) centered
convergence to 0} holds in this case, too. Summing \eqref{eq:LT
for SN with deterministic response} and \eqref{eq:Y(ut) centered
convergence to 0} finishes the proof for this subcase because
$$	c(t)^2h(t)^2	~\sim~	\int_0^t v(y){\rm d}y+c(t)^2h(t)^2.	$$

\noindent{\it Subcase} $p\in(0,1)$: We only give a proof in the
case $\sigma^2<\infty$, the other cases being similar. Relation
\eqref{eq:convergence A3} then reads
\begin{equation}\label{eq:convergence_A3_eq}
{Y(ut)-\mu^{-1}\int_0^{ut}h(y)\dy\over \sigma \sqrt{t}h(t)} ~\fdc~
c_1 V_\beta(u)+ c_2 \int_0^u(u-y)^\rho \, {\rm d}
\mathcal{S}_\alpha(y),
\end{equation}
where $c_1:=\sqrt{\frac{(1-p)(1+\beta)}{p\mu}}$ and
$c_2:=\mu^{-(\alpha+1)/\alpha}$. Write
\begin{eqnarray*}
\frac{Y(ut)-\mu^{-1}\int_0^{ut}h(y)\dy}{\sigma \sqrt{t}h(t)} & = &
\frac{Y(ut)-\sum_{k\geq 0}h(ut-S_k)\1_{\{S_k\leq ut\}}}{\sigma \sqrt{t}h(t)}   \\
& &
+ \frac{\sum_{k\geq 0}h(ut-S_k)\1_{\{S_k \leq ut\}}-\mu^{-1}\int_0^{ut}h(y)\dy}{\sigma \sqrt{t}h(t)}   \\
&=:& A_t(u) + B_t(u).
\end{eqnarray*}
According to Proposition \ref{Prop:mu<infty}, \eqref{eq:LT for
martingale term} holds which is equivalent to $$A_t(u)\quad
\overset{\mathrm{f.d.}}{\Rightarrow}\quad  c_1 V_\beta(u).$$ From
\eqref{eq:LT for SN with deterministic response} we already know
that
\begin{equation}\label{eq:LT for
B_t(u)} B_t(u)\quad \overset{\mathrm{f.d.}}{\Rightarrow}\quad
c_2\int_{[0,u]}(u-y)^\rho \, {\rm d} \mathcal{S}_2(y).
\end{equation}
By the Cram\'er-Wold device and L\'evy's continuity theorem,	
in order to prove \eqref{eq:convergence_A3_eq} it suffices to check
that, for any $m\in\mn$, any real numbers $\alpha_1,\ldots,
\alpha_m$, $\beta_1,\ldots, \beta_m$, any $0 < u_1 <\ldots,
u_m<\infty$ and any $w,z \in\mr$,
\begin{align}
\lim_{t \to \infty} & \me \bigg[\exp\bigg({\rm i}w\sum_{j=1}^m
\alpha_j A_t(u_j)+{\rm i}z  \sum_{r=1}^m \beta_rB_t(u_r)\bigg)\bigg]  \notag  \\
& =~    \me \Big[ \exp\bigg({\rm i}wc_1\sum_{j=1}^m \alpha_j
V_\beta(u_j)\bigg)\Big] \me \bigg[\exp\bigg({\rm
i}zc_2\sum_{r=1}^m
\beta_r \int_{[0,u_r]}(u_r-y)^\rho \, {\rm d}\mathcal{S}_2(y)\bigg)\bigg]   \notag  \\
& = \exp\big(-D(u_1,\ldots, u_m)c_1^2w^2/2\big)  \me
\bigg[\exp\bigg({\rm i}zc_2\sum_{r=1}^m
\beta_r\int_{[0,u_r]}(u_r-y)^\rho \, {\rm
d}\mathcal{S}_2(y)\bigg)\bigg] \label{eq:suffices A13}
\end{align}
with $D(u_1,\ldots,u_m)$ defined in \eqref{D_definition}.

The idea behind the subsequent proof is that while the $B_t$ is
$\F$-measurable, the finite-dimensional distributions of the $A_t$
converge weakly conditionally on $\F$. To make this precise, we
write
\begin{align*}
\me_\F & \bigg[\exp\bigg({\rm i}w\sum_{j=1}^m \alpha_jA_t(u_j) + {\rm i}z\sum_{r=1}^m \beta_rB_t(u_r)\bigg) \bigg]   \\
&=~ \exp\bigg({\rm i}z\sum_{r=1}^m \beta_rB_t(u_r)\bigg) \, \me_\F
\bigg[\exp\bigg({\rm i}w\sum_{j=1}^m \alpha_jA_t(u_j)\bigg)\bigg].
\end{align*}
In view of \eqref{eq:LT for B_t(u)}
\begin{align*}
\exp\bigg(&{\rm i}z\sum_{r=1}^m \beta_r B_t(u_r)\bigg)
~\overset{\mathrm{d}}{\to} \ \exp\bigg({\rm i}zc_2\sum_{r=1}^m
\beta_r\int_{[0,u_r]}(u_r-y)^\rho \, {\rm
d}\mathcal{S}_2(y)\bigg).
\end{align*}
Since $X$ and $\xi$ are assumed independent, relations
\eqref{eq:mgale CLT1} and \eqref{eq:mgale CLT2} read
\begin{align*}       \label{eq:mgale CLT11}
\sum_{k\geq 0} \me_\F [Z_{k+1,t}^2] ~\stackrel{\mmp}{\to}~
D(u_1,\ldots,u_m)
\end{align*}
and
\begin{equation*}    \label{eq:mgale CLT21}
\sum_{k\geq 0} \me_\F
\big[Z_{k+1,t}^2\1_{\{|Z_{k+1,t}|>y\}}\big]
~\stackrel{\mmp}{\to}~ 0
\end{equation*}
for all $y>0$, respectively. With these at hand and noting that
$$y(t):={\sqrt{\mu^{-1}tv(t)}\over \sigma\sqrt{t}h(t)}\quad\to\quad c_1,$$ we infer
\begin{eqnarray*}
\me_\F \bigg[ \exp\bigg({\rm i}w\sum_{j=1}^m \alpha_jA_t(u_j)
\bigg)\bigg]&=&\me_\F \bigg[ \exp\bigg({\rm i}w y(t)\sum_{k\geq 0}
Z_{k+1,t}\bigg)\bigg]\\
 &\overset{\mathrm{d}}{\to}&
\exp(-D(u_1,\ldots,u_m)c_1^2 w^2/2)
\end{eqnarray*}
by formula \eqref{conditional} of Lemma \ref{essential2}. Since
the right-hand side of the last expression is non-random,
Slutsky's lemma implies
\begin{align*}
\exp&\bigg({\rm i}z\sum_{r=1}^m \beta_r B_t(u_r)\bigg)
\, \me_\F \bigg[\exp\bigg({\rm i}w\sum_{j=1}^m \alpha_j A_t(u_j)\bigg) \bigg]  \\
&\overset{\mathrm{d}}{\to} \exp \bigg({\rm i}zc_2 \sum_{r=1}^m \beta_r
\int_{[0,u_r]} \! (u_r-y)^\rho \, {\rm d}\mathcal{S}_2(y) \bigg)\exp(-D(u_1,\ldots, u_m)c_1^2 w^2/2).
\end{align*}
Invoking the Lebesgue dominated convergence theorem completes the proof of \eqref{eq:suffices A13}.
\end{proof}

\begin{proof}[Proof of Theorem \ref{Thm:mu=infty}]
\noindent {\it Case} $q=0$:
According to Proposition \ref{Prop:mu=infty}
\begin{equation}    \label{eq:limit relation when mu=infty}
\sqrt{\frac{\mmp\{\xi>t\}}{v(t)}}\bigg(Y(ut)-\sum_{k\geq
0}h(ut-S_k)\1_{\{S_k\leq ut\}}\bigg) ~\fdc~ Z_{\alpha,\beta}(u).
\end{equation}
It remains to show that
\begin{equation*}
\sqrt{\frac{\mmp\{\xi>t\}}{v(t)}} \sum_{k\geq
0}h(ut-S_k)\1_{\{S_k\leq ut\}} ~\fdc~0.
\end{equation*}
Invoking the Cram\'{e}r-Wold device, Markov's inequality and the
regular variation of the normalization factor, we conclude that it
is enough to prove that
\begin{equation}    \label{eq:remainder to 0 when mu=infty}
\sqrt{\frac{\mmp\{\xi>t\}}{v(t)}} \me \bigg[\sum_{k\geq
0}|h(t-S_k)|\1_{\{S_k\leq t\}} \bigg] ~=~
\sqrt{\frac{\mmp\{\xi>t\}}{v(t)}} \int_{[0,t]}|h(t-x)| \, {\rm
d}U(x) ~\to~    0.
\end{equation}
This follows immediately from Lemma \ref{Lem:convergence of 1st
moment at t=1}(b) with $\phi_1(t)=|h(t)|$,
$\phi(t)=\sqrt{v(t)\mmp\{\xi>t\}}$, $\gamma=(\beta-\alpha)/2$ and
$q(t)=\sqrt{u(t)}$ for $u(t)$ defined in Proposition
\ref{Prop:mu=infty}. Note that $\phi_1=o(\phi)$ in view of the
assumption $q=0$. The proof for this case is complete because
$${\mmp\{\xi>t\}\over \sqrt{v(t)\mmp\{\xi>t\}+h(t)^2}} \ \sim \sqrt{\mmp\{\xi>t\}\over v(t)}.$$

\noindent {\it Case} $q=1$: Using Theorem 1.1 in
\citep{Iksanov:2013} when $\rho>0$\footnote{ In Theorem 1.1 of
\citep{Iksanov:2013} functional limit theorems were proved under
the assumption that $h$ is eventually nondecreasing. The latter
assumption is not needed for weak convergence of
finite-dimensional distributions which can be seen by mimicking
the proof of Theorem 2.9 in \citep{Iksanov+Marynych+Meiners:2014}.}
and Theorem 2.9 in \citep{Iksanov+Marynych+Meiners:2014} when $\rho
\in [-\alpha,0]$, we infer
\begin{equation*}
\frac{\mmp\{\xi>t\}}{h(t)} \sum_{k\geq 0}h(ut-S_k)\1_{\{S_k\leq
ut\}} ~\fdc~ \int_{[0,u]}(u-y)^\rho {\rm d}
W^\leftarrow_\alpha(y).
\end{equation*}
It remains to show that
\begin{equation*}
\frac{\mmp\{\xi>t\}}{h(t)} \bigg(Y(ut)-\sum_{k\geq 0}h(ut-S_k) \1_{\{S_k\leq ut\}}\bigg) ~\fdc~0.
\end{equation*}
Appealing to Markov's inequality and the Cram\'{e}r-Wold device we conclude that it suffices to prove
\begin{align*}
\bigg(\frac{\mmp\{\xi>t\}}{h(t)}\bigg)^{\!\!2} \, \me \bigg[\bigg(Y(ut)-\sum_{k\geq 0}h(ut-S_k) \1_{\{S_k\leq ut\}}\bigg)^{\!\!2}\bigg]      \\
~=~ \bigg(\frac{\mmp\{\xi>t\}}{h(t)}\bigg)^{\!\!2} \, \int_{[0,t]}v(t-y)
\, {\rm d}U(y) ~\to~   0.
\end{align*}
This immediately follows from Lemma \ref{Lem:convergence of 1st
moment at t=1}(b) with $\phi_1(t)=v(t)$,
$\phi(t)=h(t)^2/\mmp\{\xi>t\}$, $\gamma=2\rho+\alpha$ and
$q(t)=w(t)^2$. Note that $\phi_1=o(\phi)$ in view of the
assumption $q=1$. The proof for this case is complete because
(trivially)
$${\mmp\{\xi>t\}\over \sqrt{v(t)\mmp\{\xi>t\}+h(t)^2}} \ \sim {\mmp\{\xi>t\}\over h(t)}.$$

\noindent {\it Case} $q\in(0,1)$:
Put
\begin{eqnarray*}
\bar{A}_t(u)  & := &  \sqrt{\frac{\mmp\{\xi>t\}}{v(t)}} \sum_{k\geq 0} \big(X_{k+1}(ut-S_k)-h(ut-S_k)\big)\1_{\{S_k\leq ut\}},   \\
\bar{B}_t(u)  & := &  \sqrt{{\mmp\{\xi>t\}\over v(t)}} \sum_{k\geq
0} h(ut-S_k)\1_{\{S_k\leq ut\}}
\end{eqnarray*}
and
\begin{equation*}
A_{\alpha,\beta}(u) ~:=~ q^{1/2}(1-q)^{-1/2}
\int_{[0,u]}(u-y)^{(\beta-\alpha)/2} \, {\rm d}
W^\leftarrow_\alpha(y).
\end{equation*}
We shall prove that
\begin{equation*}
\sum_{j=1}^m \gamma_j (\bar{A}_t(u_j)+\bar{B}_t(u_j))
~\stackrel{\mathrm{d}}{\to}~ \sum_{j=1}^m
\gamma_j(Z_{\alpha,\beta}(u_j)+A_{\alpha,\beta}(u_j))
\end{equation*}
for any $m\in\mn$, any $\gamma_1,\ldots, \gamma_m\in\mr$ and any
$0<u_1<\ldots < u_m<\infty$.

Set $$\bar{Z}_{k+1,t} := \sqrt{\mmp\{\xi>t\}/v(t)}\sum_{j=1}^m
\gamma_j (X_{k+1}(u_jt-S_k)-h(u_jt-S_k))\1_{\{S_k\leq
u_jt\}},\; k\in\mn_0,\; t>0.$$ Then $\sum_{j=1}^m \gamma_j
\bar{A}_t(u_j)=\sum_{k\geq 0}\bar{Z}_{k+1,t}$ and
\begin{align*}
\sum_{k\geq 0} \me_\F [\bar{Z}_{k+1,t}^2] & =
{\mmp\{\xi>t\}\over v(t)}\int_{[0,u_m]}\bigg(\sum_{j=1}^m \gamma^2_j v(t(u_j-y))\1_{[0,u_j]}(y)   \\
&\hphantom{={\mmp\{\xi>t\}\over v(t)}}
+2\sum_{1\leq r<l\leq m}\gamma_r\gamma_l
f(t(u_r-y),t(u_l-y))\1_{[0,u_r]}(y)\bigg) \, {\rm d}\nu(ty).
\end{align*}

With this at hand, we write
\begin{align}\label{equality}
\me_{\F} & \bigg[\!\exp\!\bigg(\!{\rm
i}z\sum_{j=1}^m\gamma_j\big(\bar{A}_t(u_j)+\bar{B}_t(u_j)\big)\!\bigg)\!\bigg]
= \exp\!\bigg(\!{\rm i}z\sum_{j=1}^m\gamma_j\bar{B}_t(u_j)\!\bigg)
\me_{\F} \bigg[\!\exp\!\bigg(\!{\rm i}z\sum_{k\geq 0}\bar{Z}_{k+1,t} \!\bigg) \!\bigg]  \notag  \\
& =~ \exp\!\bigg(\!{\rm
i}z\sum_{j=1}^m\gamma_j\bar{B}_t(u_j)\!\bigg) \bigg(\!\me_{\F}
\bigg[\!\exp \! \bigg(\!{\rm i}z\sum_{k\geq
0}\bar{Z}_{k+1,t}\!\bigg)\!\bigg] -\exp
\!\bigg(-\sum_{k\geq 0}\me_{\F} [\bar{Z}^2_{k+1,t}] z^2/2\!\bigg)\bigg)  \notag \\
&\hphantom{=~}+ \exp \!\bigg(\!{\rm i}z\sum_{j=1}^m\gamma_j\bar{B}_t(u_j)-\sum_{k\geq 0}\me_{\F}
[\bar{Z}^2_{k+1,t}] z^2/2\!\bigg)
\end{align}
for $z\in\mr$.

By Formula \eqref{0010000} of Lemma \ref{principal} (with
$b=q^{-1}(1-q)$)
\begin{align}\label{inter11}
\lambda_1 \sum_{j=1}^m & \gamma_j\bar{B}_t(u_j)+\lambda_2\sum_{k\geq 0} \me_\F [\bar{Z}_{k+1,t}^2]\notag	\\
& = \lambda_1 \sqrt{{\mmp\{\xi>t\}\over v(t)}} \int_{[0,u_m]}\sum_{j=1}^m \gamma_j h(t(u_j-y))\1_{[0,u_j]}(y) {\rm d}\nu(ty)\notag	\\
& \hphantom{=} +\lambda_2 \frac{\mmp\{\xi>t\}}{v(t)}\int_{[0,u_m]}\bigg(\sum_{j=1}^m \gamma^2_j v(t(u_j-y))\1_{[0,u_j]}(y)	\notag	\\
& \hphantom{= +\lambda_2 \frac{\mmp\{\xi>t\}}{v(t)}\int_{[0,u_m]}\bigg(}
+2\sum_{1\leq r< l \leq m}\gamma_r\gamma_l f(t(u_r-y),t(u_l-y))\1_{[0,u_r]}(y)\bigg) \, {\rm d}\nu(ty)	\notag	\\
& \stackrel{\mathrm{d}}{\to}~ \lambda_1 \sum_{j=1}^m\gamma_jA_{\alpha,\beta}(u_j)+\lambda_2 D_{\alpha,\beta}(u_1,\ldots, u_m)
\end{align}
for any real $\lambda_1$ and $\lambda_2$ with
$D_{\alpha,\beta}(u_1,\ldots, u_m)$ defined in \eqref{xxx}.
Hence,
\begin{align*}
\exp \! \bigg(\! & {\rm i}z \sum_{j=1}^m\gamma_j\bar{B}_t(u_j)-\sum_{k\geq 0}\me_{\F} [\bar{Z}^2_{k+1,t}| z^2/2\!\bigg)	\\
&\stackrel{\mathrm{d}}{\to}~
\exp\bigg({\rm i} z \sum_{j=1}^m\gamma_jA_{\alpha,\beta}(u_j)-D_{\alpha,\beta}(u_1,\ldots,u_m)z^2/2\!\bigg)
\end{align*}
for each $z\in\mr$, and thereupon
\begin{align*}
\lim_{t \to \infty} \me & \bigg[\!\exp\!\bigg(\!{\rm i}z \!
\sum_{j=1}^m\gamma_j\bar{B}_t(u_j)
-\sum_{k\geq 0}\me_{\F} [\bar{Z}^2_{k+1,t}] z^2/2 \! \bigg)\!\bigg]    \\
& =~ \me \bigg[\!\exp\!\bigg(\!{\rm i}z \!\sum_{j=1}^m \gamma_j
A_{\alpha,\beta}(u_j)-D_{\alpha,\beta}(u_1,\ldots, u_m) z^2/2 \!
\bigg)\!\bigg]\\&=~ \me \bigg[\!\exp\!\bigg(\!{\rm i}z
\!\sum_{j=1}^m \gamma_j
(A_{\alpha,\beta}(u_j)+Z_{\alpha,\beta}(u_j))\! \bigg)\!\bigg]
\end{align*}
by Lebesgue's dominated convergence theorem, the second equality
following from the fact that $\sum_{j=1}^m \gamma_j
Z_{\alpha,\beta}(u_j)$ is centered normal with variance
$D_{\alpha,\beta}(u_1,\ldots,u_m)$.

According to Formula \eqref{probab} of Lemma \ref{essential2}
$$\me_{\F} \bigg[\!\exp \!
\bigg(\!{\rm i}z\sum_{k\geq 0}\bar{Z}_{k+1,t}\!\bigg)\!\bigg]
-\exp \!\bigg(-\sum_{k\geq 0}\me_{\F} [\bar{Z}^2_{k+1,t}]
z^2/2\!\bigg) ~ \stackrel{\mathrm{\mmp}}{\to}~ 0.$$
Hence the first summand on the right-hand side of \eqref{equality} tends to
zero in probability if we verify that
\begin{equation}\label{0010121} \sum_{k\geq 0} \me_\F
[\bar{Z}_{k+1,t}^2]~\stackrel{\mathrm{d}}{\to}
~D_{\alpha,\beta}(u_1,\ldots,u_m)
\end{equation}
and
\begin{equation}\label{0008121}
\sum_{k\geq 0} \me_\F
[\bar{Z}_{k+1,t}^2\1_{\{|\bar{Z}_{k+1,t}|>y\}}]~\stackrel{\mathrm{\mmp}}{\to}~
0
\end{equation}
for all $y>0$. Relation \eqref{0010121} follows from
\eqref{inter11} with $\lambda_1=0$ and $\lambda_2=1$.
In view of the inequality \eqref{for_ref} relation \eqref{0008121} is implied by
\eqref{inter1} which has already been checked.
This finishes the proof for this case because
$\lim_{t \to \infty} \frac{\mmp\{\xi>t\}v(t)}{\mmp\{\xi>t\}v(t)+h(t)^2}=1-q$ ensures that
$$\sqrt{1-q}\sqrt{\mmp\{\xi>t\}\over v(t)}\quad \sim\quad \frac{\mmp\{\xi>t\}}{\sqrt{\mmp\{\xi>t\}v(t)+h(t)^2}}.	$$
The proof of Theorem \ref{Thm:mu=infty} is complete.
\end{proof}

\subsection{Proofs of Corollary \ref{Cor:J_1 convergence} and Proposition \ref{main3}}

\begin{proof}[Proof of Corollary \ref{Cor:J_1 convergence}]
We first show that the function $f(u,w)=\me [X(u)X(w)] -\me[X(u)]
\me[X(w)]$ is uniformly regularly varying in strips in $\mr_+^2$
of index $\beta$ with limit function $C$.

The assumption $q<1$ ensures
$$\lim_{t \to \infty} v(t)/h(t)^2 = \infty,$$
hence $\me [X(t)^2] \sim v(t)$, in particular, $v(t)$ is regularly
varying of index $\beta$ 
which must be nonnegative. Further, $\lim_{t \to \infty} \me [X(ut)] \me [X(wt)]
/ v(t)=0$ because
\begin{equation*}
\frac{\me [X(ut)]\me[X(wt)]}{v(t)} ~\leq~ {(\me [X(wt)])^2\over
v(wt)} {v(wt)\over v(t)}
\end{equation*}
for $0<u<w$ by monotonicity.
More importantly, $\lim_{t \to \infty} \me [X(ut)X(wt)]/v(t)=C(u,w)$, and the function $C$ is
continuous in $\mr^2_+$ (see Lemma 2 in \citep{Haan+Resnick:1979})
because, for each $t>0$, the function $(u,w) \mapsto \me
[X(ut)X(wt)]$ is increasing in each variable.
Recall that convergence of monotone functions to a continuous
limit is necessarily locally uniform. Therefore, in both
limit relations above the convergence is locally uniform in $\mr^2_+$. Hence
\begin{equation*}
\lim_{t \to \infty} \frac{f(ut,wt)}{v(t)}   ~=~ C(u,w)
\end{equation*}
locally uniformly in $\mr^2_+$ which entails the uniformity in
strips, as desired.

Recall that $\beta\geq 0$ and note that whenever $h$ is regularly
varying of index $\rho$ we must have $\rho\geq 0$. Putting
\begin{equation*}
Q_{\alpha,\rho}(u)  ~:=~    \int_{[0,u]}(u-y)^\rho \, {\rm
d}W^\leftarrow_\alpha(y),    \quad   u\geq 0
\end{equation*}
we observe that $Q_{\alpha,\rho}:=(Q_{\alpha,\rho}(u))_{u \geq
0}$ is a.s.\ continuous on $[0,\infty)$ with
$Q_{\alpha,\rho}(0)=0$ (in the case $\rho=0$ the process is just
$W^\leftarrow_\alpha$ and the random function $u \mapsto
W^\leftarrow_\alpha(u)$ is a.s.~continuous as the generalized
inverse of the a.s.~strictly increasing random function $t \mapsto
W_\alpha(t)$). Now we check that, by continuity, we can define
$Z_{\alpha,\beta}(0)$ to be equal to $0$. To this end, observe that
\begin{equation*}
\me [Z_{\alpha,\beta}(u)^2]
~=~ \me \bigg[\int_{[0,u]}(u-y)^\beta \, {\rm d}W^\leftarrow_\alpha(y)\bigg]
~=~ \frac{\Gamma(\beta+1)}{\Gamma(1-\alpha)\Gamma(\alpha+\beta+1)} u^{\beta+\alpha}
~\to~	0,	\quad u\downarrow 0
\end{equation*}
having used \eqref{usef} for the second equality.
Hence $\lim_{u\downarrow 0} Z_{\alpha,\beta}(u)=0$ in probability.
An important consequence of the fact that the limit processes are equal to zero at the origin
is that the weak convergence of the finite-dimensional distributions
proved in Theorem \ref{Thm:mu=infty} for $u>0$ can be extended to
$u\geq 0$.

Since, for each $t>0$, the process $(Y(ut))_{u \geq 0}$ is
a.s.~nondecreasing, according to Theorem 3 in \citep{Bingham:1971}
it remains to show that the limit processes are continuous in
probability. This is obvious for $Q_{\alpha,\rho}$. Further,
\begin{align*}
\mmp&\big\{|Z_{\alpha,\beta}(w)-Z_{\alpha,\beta}(u)|>\varepsilon\big|W^\leftarrow_\alpha\big\} \leq \frac{1}{\varepsilon^2} \bigg(\! \int_{[0,u]} \!
(u-y)^\beta \, {\rm d}W^\leftarrow_\alpha(y)\\
&\hspace{3cm}+ \int_{[0,w]} \!
(w-y)^\beta \, {\rm d}W^\leftarrow_\alpha(y) -2\int_{[0,u]} \!
C(u-y,w-y) \, {\rm d}W^\leftarrow_\alpha(y)\!\bigg)
\end{align*}
for $0<u<w$ and $\varepsilon>0$, by Chebyshev's inequality. As $w
\downarrow u$, the second term converges a.s.\ to
$\int_{[0,u]}(u-y)^\beta{\rm d}W^\leftarrow_\alpha(y)$ in view
of the aforementioned a.s.~continuity. By Fatou's lemma
\begin{equation*}
\liminf_{w \downarrow u} \int_{[0,u]} \!\! C(u-y,w-y) \, {\rm
d}W^\leftarrow_\alpha(y) \geq \int_{[0,u]} \!\! C(u-y,u-y) \,
{\rm d}W^\leftarrow_\alpha(y) = \int_{[0,u]} \! (u-y)^\beta{\rm
d}W^\leftarrow_\alpha(y)
\end{equation*}
as $C$ is continuous in $\mr^2_+$. Hence, $\lim_{w \downarrow u}
\mmp\big\{|Z_{\alpha,\beta}(u)-Z_{\alpha,\beta}(w)| \geq
\varepsilon\big|W^\leftarrow_\alpha \big\}=0$ a.s. The proof of
this convergence when $w\uparrow u$ is analogous. Applying now the
Lebesgue dominated convergence theorem we conclude that
$(Z_{\alpha,\beta}(u))_{u \geq 0}$ is continuous in probability.
The proof of Corollary \ref{Cor:J_1 convergence} is complete.
\end{proof}

\begin{proof}[Proof of Proposition \ref{main3}]
Let $Z^\ast_\alpha$ be a version of $Z_{\alpha,\alpha}$. We show
that for every interval $[a,b]$ with $0<a<b$,
\begin{equation}    \label{5665*}
\me \bigg[\underset{t\in
[W_\alpha(a),W_\alpha(b)]}{\sup}\,Z^{\ast}_\alpha(t)^2 \, \Big|
\, W^\leftarrow_\alpha \bigg] = \infty \quad\text{a.s.}
\end{equation}
To prove this, first notice that according to Theorem 2 in \citep{Fristedt:1979} there
exists an event $\Omega^\prime$ with $\mmp(\Omega^\prime)=1$ such
that for any $\omega\in \Omega^\prime$
$$  \limsup_{y\uparrow s}\frac{W_\alpha (s,\omega)-W_{\alpha}(y,\omega)}{(s-y)^{1/\alpha}}\leq r $$
for some deterministic constant $r\in (0,\infty)$ and some
$s:=s(\omega)\in [a,b]$. Fix any $\omega\in \Omega^\prime$. There
exists $s_1:=s_1(\omega)$ such that
$$  \big(W_{\alpha}(s,\omega)-W_{\alpha}(y,\omega)\big)^{-\alpha} \geq (s-y)^{-1} r^{-\alpha}/2    $$
whenever $y\in (s_1, s)$.
Set $t:=t(\omega)=W_{\alpha}(s,\omega)$ and write
\begin{eqnarray*}
\me \big[Z^{\ast}_\alpha(t)^2
\big|W^\leftarrow_\alpha\big](\omega) & = &
\int_{[0,t(\omega)]}(t(\omega)-y)^{-\alpha} \, {\rm d} W^{\leftarrow}_\alpha(y,\omega)   \\
& = &
\int_{[0,W_\alpha(s,\omega)]}(W_{\alpha}(s,\omega)-y)^{-\alpha} \, {\rm d}W^{\leftarrow}_\alpha(y,\omega)  \\
& = &
\int_0^s\big(W_{\alpha}(s,\omega)-W_\alpha(y,\omega)\big)^{-\alpha} \, {\rm d}y    \\
& \geq & \int_{s_1}^s \big(W_\alpha(s,\omega)-W_\alpha(y,\omega)\big)^{-\alpha} \, {\rm d}y    \\
& \geq & {1\over 2r^{\alpha}} \int_{s_1}^s (s-y)^{-1} \, {\rm d}y
~=~ +\infty.
\end{eqnarray*}
This proves \eqref{5665*}, for $t(\omega) \in [W_\alpha(a,\omega),W_\alpha(b,\omega)]$
for all $\omega\in \Omega^\prime$.

Now observe that if $Z^\ast_\alpha$ has paths in $D(0,\infty)$
a.s., then, for any $0<a<b$,
\begin{equation}    \label{1232}
\mmp\Big\{\Big|\sup_{t \in [W_\alpha(a),W_\alpha(b)]}
Z^\ast_\alpha(t) \Big| < \infty \, \Big| \, W^\leftarrow_\alpha
\Big\} = 1.
\end{equation}
Note that the process $W_\alpha$ is measurable with respect to the
$\sigma$-field generated by $W^\leftarrow_\alpha$ and that, given
$W^\leftarrow_\alpha$, the process $Z_\alpha^\ast$ is centered
Gaussian. Hence, from Theorem 3.2 on p.~63 in \citep{Adler:1990}
(applied to $(Z_\alpha^\ast(t))_{t \in
[W_\alpha(a),W_\alpha(b)]}$ and $(-Z_\alpha^\ast(t))_{t \in
[W_\alpha(a),W_\alpha(b)]}$ both conditionally given
$W^\leftarrow_\alpha$), we conclude that \eqref{1232} is
equivalent to
\begin{equation}    \label{5665}
\me \bigg[\underset{t\in
[W_\alpha(a),W_\alpha(b)]}{\sup}\,Z^{\ast}_\alpha(t)^2 \, \Big|
\, W^\leftarrow_\alpha \bigg] <\infty \quad\text{a.s.}
\end{equation}
which cannot hold due to \eqref{5665*}. Hence $Z^\ast_\alpha$ has
paths in $D(0,\infty)$ with probability less than $1$.

Finally, suppose that $C(u,w)=0$ for all $u\neq w$, $u,w>0$. Then,
given $W_\alpha^\leftarrow$, the Gaussian process $Z^\ast_\alpha$
has uncorrelated, hence independent values. For any fixed $t>0$
and any decreasing sequence $(h_n)_{n \in \mn}$ with $\lim_{n \to
\infty} h_n=0$ we infer
\begin{equation}    \label{zero}
\mmp\big\{Z^\ast_\alpha \quad\text{is right-continuous at} \
t\big|W_\alpha^\leftarrow\big\} \leq
\mmp\Big\{\limsup_{n\to\infty} Z^\ast_\alpha(t+h_n) =
Z^\ast_\alpha(t) \big| W_\alpha^\leftarrow\Big\} =   0   \ \
\text{a.s.}
\end{equation}
which proves that $Z^\ast_\alpha$ has paths in the Skorokhod space
with probability $0$. To justify \eqref{zero} observe that, given
$W_\alpha^\leftarrow$, the distribution of $Z^\ast_\alpha(t)$ is
Gaussian, hence continuous, while $\limsup_{n\to\infty}
Z^\ast_\alpha(t+h_n)$ is equal to a constant (possibly $\pm
\infty$) a.s.~by the Kolmogorov zero-one law which is applicable
because $Z^\ast_\alpha(t+h_1)$, $Z^\ast_\alpha(t+h_2),\ldots$ are
(conditionally) independent. The proof of Proposition \ref{main3} is complete.
\end{proof}

\section{Auxiliary results}	\label{sec:Appendix}

In this section, we collect technical results some of which are known
and stated here for the reader's convenience.
Others are extensions of known results
or important technical steps used more than once in the derivations of our main results.
We begin with a series of known results:
Lemma \ref{Lem:Thm 4.2 of Billingsley} is Theorem 4.2 in \citep{Billingsley:1968},
Lemma \ref{reg_var}(a) is Theorem 1.5.2 from \citep{Bingham+Goldie+Teugels:1989},
Lemma \ref{reg_var}(b) is a consequence of Theorem 1.5.3 in \citep{Bingham+Goldie+Teugels:1989},
Lemma \ref{reg_var}(c) is Karamata's theorem (Proposition 1.5.8 in \citep{Bingham+Goldie+Teugels:1989}),
Lemma \ref{ct} is Lemma 3.2 in \citep{Iksanov+Marynych+Meiners:2014}.

\begin{lemma}   \label{Lem:Thm 4.2 of Billingsley}
Let $(S,d)$ be an arbitrary metric space. Suppose that $(Z_{un},
Z_n)$ are random elements on $S \times S$. If $Z_{un}
\Rightarrow_{n} Z_u \Rightarrow_{u} Z$ on $(S,d)$ and
\begin{equation*}
\lim_{u} \limsup_n \mmp\{d(Z_{un}, Z_n)  > \varepsilon \} ~=~ 0
\end{equation*}
for every $\varepsilon > 0$, then $Z_n \Rightarrow Z$ on $(S,d)$,
as $n\to\infty$.
\end{lemma}
\begin{lemma}\label{reg_var}
Let $g$ be regularly varying at $\infty$ of index $\rho$ and
locally bounded outside zero.

\noindent (a) Then, for all $0<a<b<\infty$,
$$	\lim_{t \to \infty} \sup_{a\leq s\leq b}\Big|\frac{g(st)}{g(t)}-s^\rho\Big| ~=~ 0.	$$

\noindent (b) Suppose $\rho\neq 0$. Then there exists a monotone
function such that $g(t)\sim u(t)$ as $t\to\infty$.

\noindent (c) Let $\rho>-1$ and $a>0$. Then $\int_a^t g(y){\rm
d}y\sim (\rho+1)tg(t)$ as $t\to\infty$.
\end{lemma}

\begin{lemma}\label{ct}
$c(t)$ appearing in \eqref{eq:FLT for N(t)} is regularly varying at $\infty$ of index $1/\alpha$.
\end{lemma}

Lemma \ref{Lem:continuous mapping D} follows from Lemma A.5 in \citep{Iksanov:2013}
in combination with the continuous mapping theorem.
We note in passing that \citep{Kurtz+Protter:1991} and Chapter VI, $\S6$c in \citep{Jacod+Shiryaev:2003}
are classical references concerning the convergence of stochastic integrals.

\begin{lemma}\label{Lem:continuous mapping D}
Let $0\leq a<b<\infty$.
\begin{itemize}
    \item[(a)]
        Suppose that, for each $t>0$, $f_t \in D$ and that the random process $(\mathcal{X}_t(y))_{a \leq y \leq b}$ has almost surely increasing
        path. Assume further that $\lim_{t \to \infty} f_t(y)=f(y)$ uniformly in $y \in [a,b]$
        and that $\mathcal{X}_t \Rightarrow \mathcal{X}$, $t\to\infty$ in the $J_1$-topology on $D[a,b]$,
        the paths of $(\mathcal{X}(y))_{a \leq y \leq b}$ being almost surely continuous. Then
        \begin{equation*}
        \int_{[a,b]} f_t(y) \, {\rm d}\mathcal{X}_t(y) ~\stackrel{\mathrm{d}}{\to} ~ \int_{[a,b]} f(y) \, {\rm d}\mathcal{X}(y),\quad t\to\infty.
        \end{equation*}
    \item[(b)]
        Assume that $\mathcal{X}_t \Rightarrow \mathcal{X}$, $t\to\infty$, in
        the $J_1$- or $M_1$-topology on $D[a,b]$ and that, as $t\to\infty$,
        finite measures $\rho_t$ converge weakly on $[a, b]$ to $\delta_c$,
        the Dirac measure concentrated at $c$. If $\mathcal{X}$ is almost surely  continuous at $c$, then
        \begin{equation*}
       \int_{[a,b]} \mathcal{X}_t(y) \, \rho_t(\dy) ~\dod~ \mathcal{X}(c),\quad t\to\infty.
        \end{equation*}
\end{itemize}
\end{lemma}
\begin{lemma}
Let $W$ be a nonnegative random variable with Laplace
transform $\varphi(s) := \me [e^{-sW}]$, $s\geq 0$. Then, for
$\theta\in (0,1)$,
\begin{equation}    \label{0023}
\me [W^{-\theta}] ~=~ \frac{1}{\Gamma(\theta)} \int_0^\infty
s^{\theta-1} \varphi(s) \, {\rm d}s.
\end{equation}
\end{lemma}
\begin{proof}
Let $R$ be a random variable with the standard exponential law
which is independent of $W$. Then, for $s \geq 0$,
$\varphi(s)=\mmp\{R/W>s\}$. Hence, $\Gamma(1+\theta) \me
[W^{-\theta}] = \me [(R/W)^\theta] = \theta \int_0^\infty
s^{\theta-1}\varphi(s) \, {\rm d}s$.
\end{proof}
\begin{lemma}\label{essential2}
Let $(Z_{k,t})_{k\in\mn,\,t>0}$ be a family of random variables
defined on some probability space $(\Omega, \mathcal{R},\mmp)$ and
let $\mg$ be a sub-$\sigma$-algebra of $\mathcal{R}$. Assume that,
given $\mg$ and for each fixed $t>0$, the $Z_{k,t}$, $k\in\mn$ are independent.
If
\begin{equation}\label{0010} \sum_{k\geq 0} \me_\mg
[Z_{k+1,t}^2] \ \dod \ D,\quad t\to\infty
\end{equation}
for a random variable $D$ and
\begin{equation}\label{0008}
\sum_{k\geq 0} \me_\mg [Z_{k+1,t}^2\1_{\{|Z_{k+1,t}|>y\}}] \
\tp\ 0, \quad	t\to\infty
\end{equation}
for all $y>0$, then, for each $z\in\mr$,
\begin{equation}\label{conditional}
\me_\mg \bigg[\exp\bigg({\rm i}z \sum_{k\geq
0}Z_{k+1,t}\bigg)\bigg] \ \dod \ \exp(-D z^2/2),\quad t\to\infty,
\end{equation}
\begin{equation}\label{char}
\me \bigg[\exp\bigg({\rm i}z \sum_{k\geq
0}Z_{k+1,t}\bigg)\bigg] \to \me \bigg[\exp\big(-Dz^2/2\big)\bigg],\quad t\to\infty
\end{equation}
and
\begin{equation}\label{probab}
\me_\mg \bigg[\exp\bigg({\rm i}z \sum_{k\geq
0}Z_{k+1,t}\bigg)\bigg] - \me_\mg \bigg[\exp\bigg({\rm i}z
\sum_{k\geq 0}\widehat{Z}_{k+1,t}\bigg)\bigg] \ \tp \ 0, \ \
t\to\infty,
\end{equation}
where, given $\mg$, $\widehat{Z}_{1,t}, \widehat{Z}_{2,t},\ldots$ are conditionally independent normal random variables with mean $0$
and variance $\me_\mg [Z^2_{k+1,t}]$, i.e.,
$$\me_\mg [\exp({\rm i}z\widehat{Z}_{k+1,t})] = \exp(-\me_\mg [Z^2_{k+1,t}] z^2/2),\ k\in\mn_0.$$
\end{lemma}
\begin{proof}
Apart from minor modifications, the following argument can be
found in the proof of Theorem 4.12 in \citep{Kallenberg:1997} in
which the weak convergence of the row sums in triangular arrays to
a normal law is investigated. For any $\varepsilon>0$,
$$\underset{k\geq 0}{\sup}\,\me_\mg [Z^2_{k+1,t}] \leq \varepsilon^2+ \underset{k\geq 0}{\sup}\,\me_\mg [Z^2_{k+1,t}\1_{\{|Z_{k+1,t}|>\varepsilon\}}]
\leq \varepsilon^2 + \sum_{k\geq 0} \me_\mg
[Z_{k+1,t}^2\1_{\{|Z_{k+1,t}|>\varepsilon\}}].  $$ Using
\eqref{0008} and letting first $t\to\infty$ and then
$\varepsilon\downarrow 0$, we infer
\begin{equation}\label{0021}
\underset{k\geq 0}{\sup}\,\me_\mg [Z^2_{k+1,t}] \ \tp \ 0.
\end{equation}
In view of \eqref{0010}
\begin{equation}\label{0016}
\me_\mg \bigg[\exp\bigg({\rm i}z \sum_{k\geq
0}\widehat{Z}_{k+1,t}\bigg) \bigg] = \exp\bigg(-\sum_{k\geq
0}\me_\mg [Z^2_{k+1,t}] z^2/2\bigg) \ \dod \ \exp(-Dz^2/2)
\end{equation}
for each $z\in\mr$. Next, we show that $\sum_{k\geq 0}Z_{k+1,t}$
has the same distributional limit as $\sum_{k\geq
0}\widehat{Z}_{k+1,t}$ as $t \to \infty$. To this end, for
$z\in\mr$, consider
\begin{align*}
\bigg|\me_\mg & \bigg[\exp\bigg({\rm i}z \sum_{k\geq
0}Z_{k+1,t}\bigg)\bigg]
-\me_\mg \bigg[\exp\bigg({\rm i}z \sum_{k\geq 0}\widehat{Z}_{k+1,t}\bigg)\bigg]\bigg| \\
&= \bigg|\prod_{k\geq 0}\me_\mg \big[\exp\big({\rm i}z
Z_{k+1,t}\big)\big]
-\prod_{k\geq 0} \me_\mg \big[\exp\big({\rm i}z \widehat{Z}_{k+1,t}\big)\big]\bigg|   \\
&\leq \sum_{k\geq 0}\Big|\me_\mg \big[\exp\big({\rm i}z
Z_{k+1,t}\big)\big]
- \me_\mg \big[\exp\big({\rm i}z \widehat{Z}_{k+1,t}\big)\big] \Big|  \\
&\leq \sum_{k\geq 0}\Big|\me_\mg \big[\exp\big({\rm i}z Z_{k+1,t}\big)\big]-1+\frac{z^2}{2} \me_\mg \big[Z^2_{k+1,t}\big]\Big|  \\
&\hphantom{\leq}+ \sum_{k\geq 0}\Big|\me_\mg \big[\exp\big({\rm
i}z \widehat{Z}_{k+1,t}\big)\big]-1
+\frac{z^2}{2} \me_\mg \big[\widehat{Z}^2_{k+1,t}\big]\Big|   \\
&\leq z^2\sum_{k\geq 0}\me_\mg \big[Z^2_{k+1,t}\big(1\wedge
6^{-1}|zZ_{k+1,t}|\big)\big] +z^2\sum_{k\geq 0}\me_\mg
\big[\widehat{Z}^2_{k+1,t}\big(1\wedge
6^{-1}|z\widehat{Z}_{k+1,t}|\big)\big],
\end{align*}
where, to arrive at the last line, we have utilized
$|\me_\mg[\cdot]| \leq \me_\mg [|\cdot|]$ and the inequality
$$  |e^{{\rm i}z}-1-{\rm i}z+z^2/2|\leq z^2\wedge 6^{-1}|z|^3, \quad	z\in\mr,$$
which can be found, for instance, in Lemma 4.14 of
\citep{Kallenberg:1997}. For any $\varepsilon\in (0,1)$ and $z \neq
0$
\begin{align*}
\sum_{k\geq 0} \me_\mg \big[Z^2_{k+1,t}\big(1\wedge
6^{-1}|zZ_{k+1,t}|\big)\big] \leq \varepsilon \sum_{k \geq 0 }
\me [Z^2_{k+1,t}] +\sum_{k\geq 0}\me_\mg
\big[Z^2_{k+1,t}\1_{\{|Z_{k+1,t}|>6\varepsilon/|z|\}}\big].
\end{align*}
Recalling \eqref{0008} and letting first $t\to\infty$ and then
$\varepsilon\downarrow 0$ give
\begin{equation*}
\sum_{k\geq 0}\me_\mg \big[Z^2_{k+1,t}\big(1\wedge
6^{-1}|zZ_{k+1,t}|\big)\big] \ \tp \ 0.
\end{equation*}
Further,
\begin{eqnarray*}
\sum_{k\geq 0} \me_\mg \big[\widehat{Z}^2_{k+1,t}\big(1 \wedge
6^{-1}|z\widehat{Z}_{k+1,t}|\big)\big] &\leq& \frac{|z|}{6}
\sum_{k\geq 0} \me_\mg \big[|\widehat{Z}_{k+1,t}|^3\big]\\
= \frac{\sqrt{2}|z|}{3\sqrt{\pi}} \sum_{k \geq 0}(\me_\mg [Z_{k+1,t}^2])^{3/2}  
&\leq& \frac{\sqrt{2}|z|}{3\sqrt{\pi}} \Big(\underset{k\geq
0}{\sup}\,\me_\mg [Z^2_{k+1,t}]\Big)^{\!1/2} \, \sum_{k\geq 0}
\me_\mg [Z^2_{k+1,t}].
\end{eqnarray*}
Here, \eqref{0010} and \eqref{0021} yield
\begin{equation*}
\sum_{k\geq 0}\me_\mg \big[\widehat{Z}^2_{k+1,t}\big(1\wedge
6^{-1}|z\widehat{Z}_{k+1,t}|\big)\big]\ \tp \ 0.
\end{equation*}
Thus, we have already proved \eqref{probab} which together with
\eqref{0016} implies \eqref{conditional}. Relation \eqref{char}
follows from \eqref{conditional} by
taking expectations and using uniform integrability. The proof of Lemma \ref{essential2} is
complete.
\end{proof}

\begin{lemma}	\label{integrals}
Let $\rho>-1$, $\alpha\in (0,1)$ and $C$ denote the limit function
for $f(u,w)=\me [X(u)X(w)]-\me [X(u)]\me [X(w)]$ wide-sense
regularly varying in $\mr_+^2$ of index $\beta$ for some
$\beta\geq -\alpha$. Then the integrals
\begin{equation}    \label{0041}
\int_{[0,s]} C(s-y,t-y) \, {\rm d}W_\alpha^\leftarrow(y),
\quad	0<s<t<\infty\quad\text{and}\quad \int_{[0,s]}(s-y)^\rho \, {\rm d}W_\alpha^\leftarrow(y),\quad s>0
\end{equation}
exist as Lebesgue-Stieltjes integrals and are almost surely
finite. Furthermore, the process $Z_{\alpha,\beta}$ is
well-defined.
\end{lemma}
\begin{proof}
To begin with, we intend to show that
\begin{equation}
\label{004222} \me \bigg[\int_{[0,s]}(s-y)^\rho \, {\rm
d}W_\alpha^\leftarrow(y)\bigg] ~<~ \infty,\quad s>0.
\end{equation}
To this end, we first derive the following identity
\begin{equation}    \label{equa}
\me [W_\alpha^\leftarrow(y)] ~=~
\frac{1}{\Gamma(1-\alpha)\Gamma(1+\alpha)} y^\alpha ~=:~ d_\alpha
y^\alpha,\quad y \geq 0.
\end{equation}
Indeed,
\begin{eqnarray*}
\me [W_\alpha^\leftarrow(y)] & = & \int_0^\infty
\mmp\{W_\alpha^\leftarrow(y)>t\} \, {\rm d}t ~=~ \int_0^\infty
\mmp\{W_\alpha(t)\leq y\} \, {\rm d}t\\
&=&\int_0^\infty \mmp\{t^{1/\alpha}W_\alpha(1)\leq y\} \, {\rm d}t=\int_0^\infty \mmp\{W_\alpha(1)^{-\alpha} y^\alpha\geq t\}\, {\rm d}t\\
&=&\me [W_\alpha(1)^{-\alpha}] y^\alpha ~=~ \frac{1}{\Gamma(1-\alpha)\Gamma(1+\alpha)} y^{\alpha}
\end{eqnarray*}
where the last equality follows from \eqref{0023} with
$\theta=\alpha$ and $\varphi(s)=\exp(-\Gamma(1-\alpha)s^\alpha)$.
Hence
\begin{eqnarray}
\me \bigg[\int_{[0,s]} (s-y)^\rho \, {\rm
d}W_\alpha^\leftarrow(y) \bigg] & = &
{\alpha\over \Gamma(1-\alpha)\Gamma(1+\alpha)} \int_0^s (s-y)^\rho y^{\alpha-1} \, {\rm d}y \notag\\
&=&{\Gamma(\rho+1)\over
\Gamma(1-\alpha)\Gamma(\rho+\alpha+1)}s^{\rho+\alpha}\label{usef}
\end{eqnarray}
which proves \eqref{004222}.

Passing to the proof of
\begin{equation}\label{0042} \me \bigg[\int_{[0,s]}C(s-y,t-y) \, {\rm
d}W_\alpha^\leftarrow(y)\bigg] ~<~ \infty,\quad 0<s<t
\end{equation}
we assume that $C(u,w)>0$ for some $u\neq w$, $u,w>0$
and then observe that in view of \eqref{equa} and \eqref{eq:impo ineq2},
\begin{eqnarray*}
\me \bigg[\int_{[0,s]} C(s-y,t-y) \, {\rm
d}W_\alpha^\leftarrow(y) \bigg] & = &
\alpha d_\alpha \int_{0}^{s} C(s-y,t-y)y^{\alpha-1} \, {\rm d}y \\
& \leq & {\alpha d_\alpha\over 2}
\int_0^s\big((s-y)^\beta+(t-y)^\beta\big)y^{\alpha-1} \, {\rm d}y
~<~ \infty
\end{eqnarray*}
since $\alpha>0$ and $\beta \geq -\alpha > -1$. This proves \eqref{0042}.

Further, we check that the process $Z_{\alpha,\beta}$ is
well-defined. To this end, we show that the function $\Pi(s,t)$
defined by
$$  \Pi(s,t) ~:=~ \int_{[0,s]}C(s-y,t-y) \, {\rm d}W^\leftarrow_\alpha(y), \quad 0<s\leq t  $$
is nonnegative definite, i.e., for any $m\in\mn$, any
$\gamma_1,\ldots, \gamma_m\in\mr$ and any
$0<u_1<\ldots<u_m<\infty$
\begin{align*}
&\sum_{j=1}^m  \gamma_j^2 \Pi(u_j,u_j)+2\sum_{1\leq r<l\leq m} \gamma_r\gamma_l \Pi(u_r, u_l)   \\
&= \sum_{i=1}^{m-1}
\int_{(u_{i-1},u_i]}\bigg(\sum_{k=i}^m\gamma_k^2 C(u_k-y,u_k-y)
+2\sum_{i \leq r < l \leq m}\gamma_r\gamma_l C(u_r-y,u_l-y)\bigg) \, {\rm d}W^\leftarrow_\alpha(y)  \\
&\hphantom{=} +\gamma_m^2 \int_{(u_{m-1},u_m]}C(u_m-y, u_m-y)
\,{\rm d}W^\leftarrow_\alpha(y) \geq 0 \quad \text{a.s.},
	\end{align*}
where $u_0:=0$. Since the second term is nonnegative a.s., it
suffices to prove that so is the first. The function $(u,w)
\mapsto C(u,w)$, $0<u\leq w$ is nonnegative definite as a limit of
nonnegative definite functions. Hence, for each $1\leq i\leq m-1$
and $y\in (u_{i-1},u_i)$,
$$\sum_{k=i}^m\gamma_k^2 C(u_k-y, u_k-y)+2\sum_{i\leq r<l\leq
m}\gamma_r\gamma_l C(u_r-y,u_l-y)\geq 0.$$ Thus, the process
$Z_{\alpha,\beta}$ does exist as a conditionally Gaussian
process with covariance function $\Pi(s,t)$, $0 < s \leq t$.
\end{proof}

Lemma \ref{principal} is designed to facilitate the proofs of
Proposition \ref{Prop:mu=infty} and Theorem \ref{Thm:mu=infty}.
\begin{lemma}\label{principal}
Suppose that condition \eqref{eq:regular_variation_inf_mean} holds
for some $\alpha\in (0,1)$ and some $\ell^\ast$, and that
$f(u,w)=\Cov [X(u)X(w)]$ is either uniformly regularly varying in
strips in $\mr_+^2$ or fictitious regularly varying in $\mr_+^2$,
in either of the cases, of index $\beta$ for some $\beta\geq
-\alpha$ and with limit function $C$. If
\begin{equation}\label{ref1} \lim_{\rho\uparrow 1} \limsup_{t\to\infty}{\mmp\{\xi>t\}\over v(t)}
\int_{(\rho z,z]}v(t(z-y)) \, {\rm d}U(ty) ~=~ 0
\end{equation}
for all $z>0$, then
\begin{align}	\label{0010000}
&\lambda_1 \sqrt{\frac{\mmp\{\xi>t\}}{v(t)}}
 \int_{[0,u_m]}\sum_{j=1}^m \gamma_j h((u_j-y)t)\1_{[0,u_j]}(y) \, {\rm d}\nu(ty)\notag	\\
&+\lambda_2 \frac{\mmp\{\xi>t\}}{v(t)}  \int_{[0,u_m]}  \bigg(\sum_{j=1}^m \gamma^2_j v((u_j-y)t)\1_{[0,u_j]}(y)	\notag	\\
&\hphantom{+\lambda_2 \frac{\mmp\{\xi>t\}}{v(t)} \int_{[0,u_m]}  \bigg(}\,
+2  \sum_{1\leq i<j\leq m}\gamma_i\gamma_j f((u_i-y)t,(u_j-y)t)\1_{[0,u_i]}(y)  \bigg) {\rm d}\nu(ty)\notag \\
&\stackrel{\mathrm{d}}{\to}~ \lambda_1b^{-1/2}
\sum_{j=1}^m\gamma_j \int_{[0,u_j]}(u_j-y)^{(\beta-\alpha)/2}\, {\rm d}W^\leftarrow_\alpha(y)	\\
&+ \lambda_2 \bigg(\sum_{j=1}^m \gamma^2_j
\int_{[0,u_j]}(u_j-y)^\beta \, {\rm d}W^\leftarrow_\alpha(y)\notag \\
&\hphantom{+\lambda_2 \frac{\mmp\{\xi>t\}}{v(t)} \int_{[0,u_m]}  \bigg(}\,+2\sum_{1\leq i<j\leq m} \gamma_i\gamma_j\int_{[0,u_i]}C(u_i-y,
u_j-y) \, {\rm d}W^\leftarrow_\alpha(y)\bigg)\notag
\end{align}
for any $m\in\mn$, any real $\gamma_1,\ldots,\gamma_m$, any
$0<u_1<\ldots<u_m<\infty$ and any real $\lambda_1$ and $\lambda_2$
provided that whenever $\lambda_1>0$
\begin{equation}\label{additional}
\lim_{t \to \infty} {v(t)\mmp\{\xi>t\}\over h(t)^2}=b\in (0,\infty)
\end{equation}
and
\begin{equation}\label{ref3}
\lim_{\rho\uparrow 1}
\limsup_{t\to\infty}\sqrt{{\mmp\{\xi>t\}\over v(t)}} \int_{(\rho
z,z]}h((z-y)t) \, {\rm d}U(ty) ~=~ 0
\end{equation}
for all $z>0$.
\end{lemma}
\begin{proof}
We only prove the lemma in the case that $\lambda_2\neq 0$ and
$C(u,w)>0$ for some $u\neq w$. Fix any $\rho \in (0,1)$ such that
$\rho u_m>u_{m-1}$ ($u_0:=0$).

Since $v$ is regularly varying at $\infty$ of index $\beta$, we infer
$$\lim_{t \to \infty} v((u-y)t)/v(t)=(u-y)^\beta,
\quad
\lit h((u-y)t)/\sqrt{v(t)\mmp\{\xi>t\}}=b^{-1/2}(u-y)^{(\beta-\alpha)/2}$$
for each $y\in [0,u)$, respectively, having utilized
\eqref{additional} for the second relation. Furthermore, the
convergence in each of these limit relations is uniform in $y\in
[0,\rho u]$ by Lemma \ref{reg_var}(a).
Since $f(u,w)$ is uniformly
regularly varying in strips in $\mr_+^2$ we conclude that for
$r<l$ the convergence $\lim_{t \to \infty}
f((u_r-y)t,(u_l-y)t)/v(t)=C(u_r-y,u_l-y)$ is uniform in $y \in [0,\rho u_r]$, too. Hence,
\begin{align*}
 \lim_{t \to \infty}& \bigg(\lambda_1 \frac{\sum_{j=1}^m \gamma_j h((u_j-y)t)\1_{[0,\rho u_j]}(y)}{\sqrt{v(t)\mmp\{\xi>t\}}}	+\lambda_2 \frac{ \sum_{j=1}^m \gamma^2_j v((u_j-y)t)\1_{[0,\rho u_j]}(y)}{v(t)}\\
&+ 2\lambda_2\frac{\sum_{1\leq r<l\leq m}\gamma_r\gamma_l f((u_r-y)t,(u_l-y)t)\1_{[0,\rho u_r]}(y)}{v(t)}\bigg)\\
&= \lambda_1b^{-1/2}\sum_{j=1}^m \gamma_j (u_j-y)^{(\beta-\alpha)/2}\1_{[0,\rho u_j]}(y)	\\
&+\lambda_2\bigg(\sum_{j=1}^m \gamma^2_j (u_j-y)^\beta \1_{[0,\rho u_j]}(y) + 2\sum_{1\leq r<l\leq m}\gamma_r\gamma_l
C(u_r-y,u_l-y)\1_{[0,\rho u_r]}(y)\bigg)
\end{align*}
uniformly in  $y\in [0,\rho u_m]$. The random function
$W^\leftarrow_\alpha$ is a.s.~continuous as has already been
explained in the proof of Corollary \ref{Cor:J_1 convergence}.
Thus\footnote{Since $y\mapsto \nu(y)$ has a.s. increasing paths,
so does $y\mapsto \mmp\{\xi>t\}\nu(ty)$ for each $t>0$.}, in view of \eqref{weak}
\begin{align*}
&\lambda_1 \sqrt{\frac{\mmp\{\xi>t\}}{v(t)}} \int_{[0,\rho u_m]}
\sum_{j=1}^m \gamma_j h((u_j-y)t)\1_{[0,\rho u_j]}(y) \, {\rm d}\nu(ty)	\notag	\\
&+\lambda_2 \frac{\mmp\{\xi>t\}}{v(t)} \int_{[0,\rho u_m]}
\bigg(\sum_{j=1}^m \gamma^2_j v((u_j-y)t)\1_{[0,\rho u_j]}(y)	\notag	\\
&\hphantom{~+\lambda_2 \frac{\mmp\{\xi>t\}}{v(t)} \int_{[0,\rho u_m]} \bigg(}
+2\sum_{1\leq r<l\leq m}\gamma_r\gamma_l f((u_r-y)t,(u_l-y)t)\1_{[0,\rho u_r]}(y)\bigg) \, {\rm d}\nu(ty)	\notag	\\
\end{align*}
\begin{align}	\label{inter}
&\stackrel{\mathrm{d}}{\to}\lambda_1b^{-1/2} \sum_{j=1}^m\gamma_j \int_{[0,\rho u_j]}(u_j-y)^{(\beta-\alpha)/2}\, {\rm d}W^\leftarrow_\alpha(y)+\lambda_2 \bigg(\sum_{j=1}^m \gamma^2_j \int_{[0,\rho u_j]}(u_j-y)^\beta \, {\rm d}W^\leftarrow_\alpha(y)\notag \\
&\hphantom{~+\lambda_2 \frac{\mmp\{\xi>t\}}{v(t)} \int_{[0,\rho u_m]} \bigg(}+~2\sum_{1\leq r<l\leq m} \gamma_r\gamma_l\int_{[0,\rho u_r]}C(u_r-y, u_l-y) \, {\rm d}W^\leftarrow_\alpha(y)\bigg)
\end{align}
by Lemma \ref{Lem:continuous mapping D}(a).
For later use note that
\begin{equation}    \label{eq:for later use}
\lim_{\rho\uparrow 1} \limsup_{t\to\infty}{\mmp\{\xi>t\}\over
v(t)} \int_{(\rho u_r,u_r]}v((u_l-y)t) \, {\rm d}U(ty) ~=~
0,\quad r<l
\end{equation}
which can be proved by the same argument as before (since
$v(t(u_l-y))/v(t)$ converges uniformly to $(u_l-y)^{\beta}$ on
$(\rho u_r,u_r]$ as $t\to\infty$), though appealing to
\begin{equation*}
\lim_{t \to \infty} \mmp\{\xi>t\}U(ty) ~=~ {y^\alpha\over
\Gamma(1-\alpha)\Gamma(1+\alpha)}
\end{equation*}
(see formula (8.6.4) on p.~361 in \citep{Bingham+Goldie+Teugels:1989}) rather than \eqref{weak}.

According to Lemma \ref{Lem:Thm 4.2 of Billingsley}, relation \eqref{0010000} follows
if we can verify that, as $\rho \uparrow 1$, the right-hand side of \eqref{inter}
converges in distribution to the right-hand side of \eqref{0010000} and that
\begin{align}    \label{0030}
\lim_{\rho \uparrow 1} \limsup_{t \to \infty} \mmp\bigg\{\bigg| & \lambda_1 \sqrt{\frac{\mmp\{\xi>t\}}{v(t)}}
\sum_{j=1}^m \gamma_j \int_{(\rho u_j,u_j]}h(t(u_j-y)) \, {\rm d}\nu(ty)\notag \\
& \!\!+\lambda_2 \frac{\mmp\{\xi>t\}}{v(t)}\bigg(\sum_{j=1}^m \gamma^2_j\int_{(\rho u_j,u_j]} v(t(u_j-y)) \, {\rm d}\nu(ty)	\notag  \\
&\hphantom{\!\!+\lambda_2 \frac{\mmp\{\xi>t\}}{v(t)}\bigg(}\notag \\
&+2\sum_{1\leq r<l\leq m}\gamma_r\gamma_l\int_{(\rho u_r,u_r]}f((u_r-y)t,(u_l-y)t) \, {\rm d}\nu(ty)\bigg)\bigg|>\delta\bigg\} ~=~ 0
\end{align}
for all $\delta>0$. The first of these (even with distributional
convergence replaced by a.s.\ convergence) is a consequence of
the monotone convergence theorem and the a.s.\ finiteness of the
integrals in \eqref{0010000} which follows from Lemma
\ref{integrals}. Left with proving \eqref{0030} use \eqref{eq:impo
ineq1} to observe that \eqref{ref1} together with \eqref{eq:for
later use} leads to
\begin{equation*}
\lim_{\rho\uparrow 1} \limsup_{t\to\infty} \frac{\mmp\{\xi>t\}}{v(t)}\int_{(\rho u_r,u_r]}|f((u_r-y)t, (u_l-y)t)| \, {\rm d}U(ty)
~=~ 0,\quad r<l.
\end{equation*}
Applying Markov's inequality we conclude that \eqref{0030} follows
from the last asymptotic relation, \eqref{ref1} and \eqref{ref3}.
This completes the proof of \eqref{0010000}.
\end{proof}

\begin{lemma}   \label{Lem:Sgibnev}
Let $0\leq r_1<r_2\leq 1$. Suppose that $\phi:[0,\infty) \to [0,\infty)$
is either increasing and $\lim_{t\to\infty}(\phi(t)/\int_0^t\phi(y)\dy)=0$,
or decreasing and, if $r_2=1$, locally integrable. 
If $\me \xi < \infty$ and $\lim_{t \to \infty} \int_{(1-r_2)t}^{(1-r_1)t} \phi(y) \dy = \infty$,
then
\begin{equation*}
\int_{[r_1t,r_2t]} \phi(t-y) \, {\rm d}U(y) ~\sim~
\frac{1}{\me\xi} \int_{(1-r_2)t}^{(1-r_1)t} \! \phi(y) \, \dy,
\quad t \to \infty.
\end{equation*}
\end{lemma}
If $\phi$ is decreasing this is Lemma 8.2 in \citep{Iksanov+Marynych+Vatutin:2013+},
the case of increasing $\phi$ and $r_1=0$, $r_2=1$ is covered by Lemma A.4 in \citep{Iksanov:2012}.
In the general case the proof goes along the lines of the proof of Theorem 4 in \citep{Sgibnev:1981},
we omit the details.

In \cite[Lemma A.4]{Iksanov:2012} it is shown that a particular
case of Lemma \ref{Lem:Sgibnev} also holds for functions $\phi$ of
bounded variation. We now give an example which demonstrates that
the result of Lemma \ref{Lem:Sgibnev} may fail to hold for
ill-behaved $\phi$. Let, for instance,
$\phi(t)=\1_{\mathbb{Q}_+^\comp}(t)$, where $\mathbb{Q}_+^\comp$
is the set of positive irrational numbers. Then $\int_0^t \phi(y)
\dy = t$. Now suppose the law of $\xi$ is concentrated at rational
points in $(0,1)$. Note that choosing these points properly,
the law of $\xi$ can be made lattice as well as nonlattice.
The points of increase of the renewal function $U(y)$ are rational points only.
Hence $\int_{[0,t]}\phi(t-y){\rm d}U(y)=0$ for rational $t$.
\begin{lemma}\label{equiv:Lemma}
Let $\me \xi<\infty$ and $\phi:[0,\infty) \to [0,\infty)$ be a
locally bounded and measurable function.

\noindent (a) Let $0\leq r_1<r_2\leq 1$. If there exists a
monotone function $\psi:[0,\infty) \to [0,\infty)$ such that
$\phi(t)\sim \psi(t)$ as $t\to\infty$, then
$$
\int_{[r_1t,r_2t]}\phi(t-y) \, {\rm d}U(y) ~\sim~
\int_{[r_1t,r_2t]}\psi(t-y) \, {\rm d}U(y), \quad t\to\infty
$$
provided that, when $r_2=1$, $\lim_{t \to \infty} \int_0^t \phi(y){\rm
d}y=\infty$ and $\lim_{t\to\infty}(\phi(t)/\int_0^t\phi(y)\dy)=0$.

\noindent (b) If there exists a locally bounded and measurable
function $\psi:[0,\infty)\to [0,\infty)$ such that
$\phi(t)=o(\psi(t))$ as $t\to\infty$ and
$\lim_{t \to \infty}\int_0^t\psi(t-y){\rm d}U(y)=+\infty$, then
$$
\int_{[0,t]}\phi(t-y) \, {\rm d}U(y)
=o\bigg(\int_{[0,t]}\psi(t-y){\rm d}U(y)\bigg),
\quad	t\to\infty.	$$
\end{lemma}
\begin{proof}
(a) For any $\delta\in (0,1)$ there exists a $t_0>0$ such that
\begin{equation}\label{ineq}
1-\delta\leq \phi(t)/\psi(t)\leq
1+\delta
\end{equation}
for all $t\geq t_0$.

\noindent {\sc Case $r_2<1$}. We have, for $t\geq
(1-r_2)^{-1}t_0$,
\begin{equation*}
(1-\delta) \int_{[r_1t,r_2t]} \!\! \psi(t-y) \, {\rm d}U(y)
~\leq~	\int_{[r_1t,r_2t]} \!\! \phi(t-y) \, {\rm d}U(y)
~\leq~ (1+\delta) \int_{[r_1t,r_2t]} \!\! \psi(t-y) \, {\rm d}U(y)
\end{equation*}
Dividing both sides by $\int_{[r_1t,r_2t]} \psi(t-y) \, {\rm d}U(y)$
and sending $t\to\infty$ and then $\delta \downarrow 0$ gives the result.

\noindent {\sc Case $r_2=1$}.
Since $\psi$ is monotone, it is locally integrable.
Further, $\lim_{t \to \infty} \int_0^t \psi(y){\rm d}y=\infty$ and
$\lim_{t\to\infty}(\psi(t)/\int_0^t\psi(y)\dy)=0$.
Hence Lemma \ref{Lem:Sgibnev} applies and
yields 
$$
\lim_{t \to \infty} \int_{[r_1t,t]} \psi(t-y) \, {\rm d}U(y)=\infty.
$$
In view of \eqref{ineq} we have
\begin{eqnarray*}
\int_{[r_1t,t]} \!\! \phi(t-y) \, {\rm d}U(y)
& \leq & (1+\delta)
\int_{[r_1t,t-t_0]} \!\! \psi(t-y) \, {\rm d}U(y)
+\int_{(t-t_0,t]} \!\! \phi(t-y) \, {\rm d}U(y)    \notag  \\
& \leq & (1+\delta) \int_{[r_1t,t]} \!\! \psi(t-y) \, {\rm
d}U(y) + U(t_0) \sup_{0 \leq y \leq t_0} \phi(y)
\end{eqnarray*}
for $t\geq (1-r_1)^{-1}t_0$, the last inequality following from
the subadditivity of $U$.  Dividing both sides by
$\int_{[r_1t,t]} \psi(t-y) \, {\rm d}U(y)$ and sending
$t\to\infty$ yields
$$
\limsup_{t\to\infty}\frac{\int_{[r_1t,t]} \! \phi(t-y) \, {\rm d}U(y)}{\int_{[r_1t,t]} \! \psi(t-y) \, {\rm d}U(y)}
~\leq~	1+\delta.
$$
The converse inequality for the lower limit follows analogously.

\noindent (b) For any $\delta\in (0,1)$ there exists a $t_0>0$
such that $\phi(t)/\psi(t)\leq \delta$ for all $t\geq t_0$. The
rest of the proof is the same as for the case $r_2=1$ of part (a).
\end{proof}

In the main text we have used the following corollary of Lemma \ref{Lem:Sgibnev}.

\begin{lemma}   \label{Lem:Sgibnev2}
Let $\me\xi < \infty$ and $\phi:[0,\infty)\to[0,\infty)$ be
locally bounded, measurable and regularly varying at $+\infty$ of
index $\beta\in (-1,\infty)$. If $\beta=0$ assume further that
there exists a monotone function $u$ such that $\phi(t) \sim u(t)
$ as $t\to\infty$. Then, for $0\leq r_1<r_2\leq 1$,
\begin{equation*}
\int_{[r_1t,r_2t]} \phi(t-y) \, {\rm d}U(y)  ~\sim~
\frac{t\phi(t)}{(1+\beta)\me
\xi}((1-r_1)^{1+\beta}-(1-r_2)^{1+\beta}), \quad t\to\infty.
\end{equation*}
\end{lemma}

The result is well known in the case where $\phi$ is increasing,
$r_1=0$, $r_2=1$, $\beta \neq 0$, and the law of $\xi$ is
nonlattice, see Theorem 2.1 in \citep{Mohan:1976}.

\begin{proof}
If $\beta \neq 0$, Lemma \ref{reg_var}(b) ensures the existence of
a positive monotone function $u$ such that $\phi(t) \sim u(t)$ as
$t \to \infty$. If $\beta=0$ such a function exists by assumption.
Modifying $u$ if needed in the right vicinity of zero we can
assume that $u$ is monotone and locally integrable. Therefore,
$$
\int_{[r_1t,r_2t]} \!\! \phi(t-y) \, {\rm d}U(y) ~\sim~
\int_{[r_1t,r_2t]} \!\! u(t-y) \, {\rm d}U(y) ~\sim~ \frac{1}{\me
\xi}\int_{(1-r_2)t}^{(1-r_1)t} \!\! u(y) \, {\rm d}y
$$
where the first equivalence follows from Lemma
\ref{equiv:Lemma}(a) and the second is a consequence of Lemma
\ref{Lem:Sgibnev} (observe that, with $g=\phi$ or $g=u$,
$\lim_{t\to\infty} (g(t)/\int_0^t g(y)\dy)=0$ and $\lim_{t \to \infty}
\int_{(1-r_2)t}^{(1-r_1)t}g(y){\rm d}y=\infty$ hold by Lemma
\ref{reg_var}(c) because $g$ is regularly varying of index
$\beta>-1$). Finally, using Lemma \ref{reg_var}(c) we obtain
\begin{eqnarray*}
{1\over \me \xi}\int_{(1-r_2)t}^{(1-r_1)t} \!\! u(y) {\rm d}y
&\sim& \frac{tu(t)}{(1+\beta)\me
\xi}((1-r_1)^{1+\beta}-(1-r_2)^{1+\beta})\\
&\sim&
\frac{t\phi(t)}{(1+\beta)\me
\xi}((1-r_1)^{1+\beta}-(1-r_2)^{1+\beta}).
\end{eqnarray*}
The proof is complete.
\end{proof}

Part (a) of the next lemma is a slight extension of Lemma 5.2 in \citep{Iksanov+Marynych+Meiners:2014}.
\begin{lemma}   \label{Lem:convergence of 1st moment at t=1}
Suppose that \eqref{eq:regular_variation_inf_mean} holds. Let
$\phi:[0,\infty) \to \mr$ be a locally bounded and measurable
function satisfying $\phi(t)\sim t^\gamma\ell(t)$ as $t\to\infty$
for some $\gamma \geq -\alpha$ and some $\ell$. If
$\gamma=-\alpha$, assume additionally that there exists a positive
increasing function $q$ such that $\lim_{t \to \infty}
\frac{\phi(t)}{\mmp\{\xi>t\}q(t)}=1$. Then
\begin{itemize}
    \item[(a)]
        \begin{equation*}
        \lim_{\rho \uparrow 1} \limsup_{t\to\infty} \frac{\mmp\{\xi>t\}}{\phi(t)} \int_{[\rho t,t]} \phi(t-y) \, {\rm d}U(y)~=~ 0;
        \end{equation*}
        in particular,
        \begin{equation*}
        \lim_{t \to \infty} \frac{\mmp\{\xi>t\}}{\phi(t)} \int_{[0,t]}\phi(t-y) \, {\rm d}U(y)
        ~=~ \frac{\Gamma(1+\gamma)}{\Gamma(1-\alpha)\Gamma(1+\alpha+\gamma)};
        \end{equation*}
    \item[(b)]
        $\int_{[0,t]}\phi_1(t-x) \, {\rm d}U(x)=o(\phi(t)/\mmp\{\xi>t\})$ as $t\to\infty$
        for any positive locally bounded function $\phi_1$ such that $\phi_1(t)=o(\phi(t))$, $t\to\infty$.
\end{itemize}
\end{lemma}
\begin{proof}
(a) In the case $\gamma\in [-\alpha, 0]$ this is just Lemma 5.2 of \citep{Iksanov+Marynych+Meiners:2014}.
In the case $\gamma>0$ exactly the same proof applies.

\noindent (b)
For any $\delta>0$ there exists a $t_0>0$ such that $\phi_1(t)/\phi(t) \leq \delta$ for all $t \geq t_0$. Hence
\begin{equation*}
\int_{[0,t]} \phi_1(t-y) \, {\rm d}U(y) ~\leq~ \delta \int_{[0,t]} \phi(t-y) \, {\rm d} U(y) +  (U(t)-U(t-t_0)) \sup_{0 \leq y \leq t_0} \phi_1(y)
\end{equation*}
for $t \geq t_0$. According to part (a) the first term on the
right-hand side grows like ${\rm const}\,\phi(t)/\mmp\{\xi>t\}$.
By Blackwell's renewal theorem, $\lim_{t \to \infty} (U(t)-U(t-t_0))=0$. Dividing the
inequality above by $\phi(t)/\mmp\{\xi>t\}$ and sending first $t
\to \infty$ and then $\delta \downarrow 0$ finishes the proof.
\end{proof}

\section*{Acknowledgements}
The authors would like to thank the anonymous referee and Prof.\
Philippe Barbe who has served as an open referee. Each of them
provided us with the most detailed and useful report out of those
we have ever gotten. In particular, the presentation of the
current versions of Theorems \ref{Thm:mu<infty} and
\ref{Thm:mu=infty} arose from their suggestions.

This research was commenced while A.\;Iksanov was visiting M\"{u}nster in May/June 2013.
Grateful acknowledgment is made for financial support and hospitality.
The research of M.\;Meiners was supported by DFG SFB 878 ``Geometry, Groups and Actions''.

\bibliography{IksMarMei_2015_1_Bib}

\end{document}